\documentclass[a4paper,12pt]{amsart}

\usepackage{amssymb,enumerate,psfrag,graphicx,amsfonts,amsrefs,amsthm,mathrsfs,amsmath,amscd,version,graphicx}
\usepackage{xcolor}
\usepackage{tikz-cd}
\usepackage{tikz}
\usetikzlibrary{arrows}
\usetikzlibrary{decorations.markings}
\usetikzlibrary{arrows.meta}

\newtheorem{theo}{Theorem}[section]
\newtheorem{theorem}[theo]{Theorem}
\newtheorem{lemma}[theo]{Lemma}
\newtheorem{proposition}[theo]{Proposition}
\newtheorem{corollary}[theo]{Corollary}
\newtheorem{definition}[theo]{Definition}

\theoremstyle{definition}
\newtheorem{remark}[theo]{Remark}
\newtheorem{example}[theo]{Example}

\begin{document}
\keywords{Gain graph, Gain-line graph, $G$-phase, oriented $G$-gain graph, Adjacency matrix, Laplacian matrix, Incidence matrix, Switching equivalence, Group representation, Fourier transform.}

\title{Gain-line graphs via $G$-phases and group representations}

\author[M. Cavaleri]{Matteo Cavaleri}
\address{Matteo Cavaleri, Universit\`{a} degli Studi Niccol\`{o} Cusano - Via Don Carlo Gnocchi, 3 00166 Roma, Italia}
\email{matteo.cavaleri@unicusano.it}

\author[D. D'Angeli]{Daniele D'Angeli}
\address{Daniele D'Angeli, Universit\`{a} degli Studi Niccol\`{o} Cusano - Via Don Carlo Gnocchi, 3 00166 Roma, Italia}
\email{daniele.dangeli@unicusano.it}

\author[A. Donno]{Alfredo Donno}
\address{Alfredo Donno, Universit\`{a} degli Studi Niccol\`{o} Cusano - Via Don Carlo Gnocchi, 3 00166 Roma, Italia}
\email{alfredo.donno@unicusano.it}

\begin{abstract}
Let $G$ be an arbitrary group. We define a gain-line graph for a gain graph $(\Gamma,\psi)$  through the choice of an incidence $G$-phase matrix inducing $\psi$.
We prove that the switching equivalence class of the gain function on the line graph $L(\Gamma)$ does not change if one chooses a different $G$-phase inducing $\psi$ or a different representative of the switching equivalence class of $\psi$.
In this way, we generalize to any group some results proven by N. Reff in the abelian case. The investigation of the orbits of some natural actions of $G$ on the set $\mathcal H_\Gamma$ of $G$-phases of $\Gamma$ allows us to characterize gain functions on $\Gamma$, gain functions on $L(\Gamma)$, their switching equivalence classes and their balance property. The use of group algebra valued matrices plays a fundamental role and, together with the matrix Fourier transform, allows us to represent a gain graph with Hermitian matrices and to perform spectral computations. Our spectral results also provide some necessary conditions for a gain graph to be a gain-line graph.
\end{abstract}

\maketitle

\begin{center}
{\footnotesize{\bf Mathematics Subject Classification (2010)}: 05C22, 05C25, 05C50, 05C76, 05E18.}
\end{center}

\section{Introduction}
\emph{Gain graphs} are  graphs where each oriented edge is labeled by an element of a group $G$ in such a way that to the opposite orientation corresponds the group inverse of the element. These labelings, that are maps from the set of oriented edges to $G$, are called \emph{gain functions}. They are usually studied up to an equivalence relation, the \emph{switching equivalence}.  More precisely, a gain graph is a pair $(\Gamma,\psi)$ where $\Gamma$ is the \emph{underlying  graph} and $\psi$ is a gain function on $\Gamma$.
For example, \emph{signed graphs} are gain graphs with $G=\{\pm1\}$, and the classical (unsigned) graphs can be seen as gain graphs on the trivial group (see  \cite{zasglos,zasbib} for a glossary and a periodically updated bibliography on the subject). \\
\indent In the recent years many works were devoted to the investigation, in the setting of gain graphs, of several topics coming from the classical case. Among them, there are graph matrices and their spectra (see \ \cite{bel1,nostro,reff0,reff1,adun,shahulgermina}), or the concept of \emph{orientation} and \emph{line graph} (see \ \cite{francesco,reff0,zas4}). All of them are inspired by the corresponding generalizations existing for signed graphs \cite{acharya,bel2,zasign,zaslavskyori,zasmat}.

In this paper we introduce and investigate the gain-line graph of a gain graph, without any assumption on the group $G$. To the best of our knowledge, the line graph has been studied only for signed graphs \cite{zasmat}, for $\mathbb T_4$-gain graphs \cite{francesco} and also  for gain graphs on abelian groups  \cite{reff0}. Our research moves from the latter.\\
\indent In \cite{reff0} Reff introduced the \emph{incidence $G$-phase functions} of a graph $\Gamma$ (in our language, simply $G$-phases), as a generalization of the \emph{bidirections} in signed graphs or, simply, as a generalization of the incidence matrices in the classical theory. For a simple graph, the incidence matrix contains all necessary information to reconstruct its adjacencies. Similarly, a $G$-phase induces a  gain function and then determines a gain graph. In analogy with \cite{zasmat}, a pair consisting of a $G$-phase and the associated gain graph was called  an \emph{oriented $G$-gain graph}. Indeed, Reff defined the concept of line graph in the context of oriented $G$-gain graphs. That is, he defined a map from $G$-phases of a graph to  $G$-phases of its line graph (this map  is denoted with $L$ in the present paper). In  \cite[Theorem~4.2]{reff0} it is proved that, when $G$ is abelian, $G$-phases inducing switching equivalent gain functions on $\Gamma$ are associated with  $G$-phases  inducing switching equivalent gain functions on its line graph $L(\Gamma)$. In other words, he managed to define the gain-line graph for switching equivalence classes of gain functions on abelian groups. The question \cite[Question~2]{reff0}, about a possible generalization of the aforementioned theorem in the non-abelian case, remained open. One of the main results of this paper is a complete answer to this question, presenting a  gain-line construction, well-posed also for non-abelian groups.

We go beyond this mere generalization as we investigate more in depth the set of $G$-phases, interpreting them as a set of group algebra valued matrices, and studying different actions of $G$ on this set. This gives a new tool to study, concurrently, (switching classes of) gain functions on a graph and (switching classes of) gain functions on its line graph (and their balance property), just looking at the orbits of these actions (see Theorem \ref{th:r}, Theorem \ref{th:l}, Theorem \ref{th:rl}, Corollary \ref{cor:bal}). Moreover, this sheds a new light on the relation between the oriented $G$-gain graphs in the sense of \cite{reff0} and the choice of a representative of the switching class of the gain-line (see Corollary \ref{cor:cap}).\\
\indent As we already mentioned, the (unoriented) incidence matrix $N$ of a graph $\Gamma$ identifies the adjacency of $\Gamma$. In particular, one has
 \begin{equation}\label{eq:1}
NN^T=\Delta_{\Gamma}^+,
\end{equation}
 where $\Delta_{\Gamma}^{+}$ is the \emph{signless Laplacian matrix} of $\Gamma$. On the other hand, the incidence matrix $N$ identifies also the adjacency of the line graph $L(\Gamma)$, as
\begin{equation}\label{eq:2}
N^TN=2I+A_{L(\Gamma)},
\end{equation}
where  $A_{L(\Gamma)}$ is the adjacency matrix of $L(\Gamma)$ (we refer the reader to the next section for formal definitions and to \cite[Eq.~7.29]{cve} or \cite{russ} for both equations). In the classical case one can then construct a line graph directly from the incidence matrix of $\Gamma$, and not necessarily from the incidence matrix of $L(\Gamma)$. An analogue is true for signed graphs (see \cite[Eq.~V.1]{zasmat}), for $\mathbb T_4$-gain graphs (see \cite[Theorem~1]{francesco}) and for $\mathbb T$-gain graphs (see \cite[Theorem~5.1]{reff0}), where $\mathbb T$ is the unit circle group and $\mathbb T_4=\{1,i,-1,-i\}$ is its subgroup consisting of the fourth roots of unity. Essentially, this is possible every time one has a complex matrix analogue of the incidence matrix.

Approaching a $G$-phase as a group algebra valued matrix gives us the advantage that the left hand sides  of Eqs.~\eqref{eq:1} and \eqref{eq:2} make sense even when the $G$-phase is not a complex matrix. For this reason in Section \ref{section4}, beside the gain graphs basics, we give some definitions and properties of the space $M_{n\times m}(\mathbb C G)$ of $n\times m$ matrices with entries in the group algebra $\mathbb C G$. We then define the \emph{adjacency matrix} $A_{\Gamma,\psi}$ and the \emph{$s$-Laplacian matrix} $\Delta^s_{\Gamma,\psi}$ of a gain graph $(\Gamma,\psi)$, where $s$ is a \emph{central weak involution} of $G$. We show that two gain functions are switching equivalent if and only if their associated adjacency matrices (or $s$-Laplacian matrices) are conjugated by a diagonal matrix with group valued diagonal entries (see Theorem \ref{th:1}). \\
\indent In Section \ref{sec:phases}, starting from a graph $\Gamma$ with $n$ vertices and $m$ edges, we define the set $\mathcal H_\Gamma$ of the $G$-phases of $\Gamma$ as a subset of $M_{n\times m}(\mathbb C G)$. From a $G$-phase $H$, we define a gain function $\Psi(H)$ on $\Gamma$ (depending on the choice of a central weak involution $s_1\in G$) and a gain function $\Psi_L(H)$ on $L(\Gamma)$ (depending on the choice of a central weak involution $s_2\in G$), in such a way that the analogues of Eqs.~\eqref{eq:1} and \eqref{eq:2} hold (see Lemma \ref{lem:r} and Lemma \ref{lem:l}). We call \emph{compatible} a pair $(\Psi(H),\Psi_L(H))$, that is, a gain function on $\Gamma$ together with a gain function on its line graph $L(\Gamma)$ coming from the same $G$-phase $H$. Actually, every gain function of $\Gamma$ can be seen as the gain function induced by some $G$-phase. More
precisely, starting from an orientation of the edges of the underlying graph $\Gamma$ and a gain function, we produce a $G$-phase inducing such a gain function and then we determine an oriented $G$-gain graph (see Proposition \ref{prop:sur}). Now, using such a $G$-phase, we are able to define a gain-line graph for any gain graph $(\Gamma,\psi)$. The already mentioned results about the orbits in $\mathcal H_\Gamma$ ensure that the switching equivalence class of the obtained gain-line graph of $(\Gamma,\psi)$ does not depend on the choice of the $G$-phase inducing $\psi$ nor on the choice of the representative of the switching class of $\psi$. Even more, compatibility of a pair depends only on the switching equivalence class of its components (see Corollary \ref{cor:nuovo}). The answer to \cite[Question~2]{reff0} follows since, by virtue of Theorem \ref{prop:line}, our gain-line graph is consistent with the  one defined in \cite{reff0}. Despite every gain function on  $\Gamma$ is induced by  elements in $\mathcal H_\Gamma$, this is not the case for  gain functions on $L(\Gamma)$. In other words, not every gain graph can be a gain-line graph, even when the necessary hypothesis that its underlying graph is a line graph is satisfied (see Example \ref{exvirus}).\\
\indent The problem of recognizing which graphs are line graphs in the classical theory has been very significant and was extensively investigated (see \cite{Bharary}). In Section \ref{sec:rep}, we take a first step in the analogous problem for general gain graphs: still following \cite{reff0}, we focus on \emph{spectral} conditions to be a gain-line graph. The main issue is that a gain graph $(\Gamma,\psi)$ does not have a canonical complex adjacency matrix. A solution comes from the \emph{matrix Fourier transform} approach, already fruitful in \cite{nostro}.
Suppose that $\Gamma$ has $n$ vertices and $\pi$ is a unitary representation  of $G$ of degree $k$; we can define the \emph{represented adjacency matrix} $\widehat{A_{\Gamma,\psi}}(\pi)\in M_{nk}(\mathbb C)$. The latter comes from the generalization, from $\mathbb C G$ to $M_n(\mathbb C G)$, of the \emph{Fourier transform} at $\pi$.  Roughly speaking, it is obtained from the adjacency matrix $A_{\Gamma,\psi}\in M_n(\mathbb C G)$ by replacing each gain entry $g\in G$ by the $k\times k$ complex
matrix $\pi(g)$. It turns out that the matrix $\widehat{A_{\Gamma,\psi}}(\pi)$ is Hermitian and we call its real spectrum the \emph{$
\pi$-spectrum of $(\Gamma,\psi)$}. In Corollary \ref{cor:1} (resp.\ Corollary \ref{cor:2}) we prove that the \emph{$\pi$-spectrum} of a gain-line graph must be contained in $[-2,\infty)$ (resp.\ in $(-\infty,2]$) if
$\pi(s_2)=I_k$ (resp.\ if $\pi(s_2)=-I_k$). In particular, when a representation $\pi$ is irreducible and unitary, the \emph{$\pi$-spectrum} of a gain-line graph cannot have simultaneously eigenvalues less than $-2$ and eigenvalues greater than $2$ (see Corollary \ref{cor:gain-line}). This new condition actually depends also on the gain function and it is not automatically satisfied when the underlying  graph is a line graph, as shown in Example~\ref{ex:qua}.\\
\indent We want to highlight that the matrix Fourier transform approach is crucial not only to define the spectrum of a gain graph via Hermitian matrices: it is also a key tool to formally include the known results about signed or $\mathbb T$-gain graphs within our more general framework, even from a matrix point of view. The complex matrices considered in the theory so far, are in fact nothing but the represented versions of the matrices of gain graphs, with a particular choice of the unitary representation and some particular choices of the central weak involutions  (see Remark \ref{rem:ponte}). This gives a more precise meaning to the statement that our paper contains and extends some of the results of \cite{francesco,reff0,zasmat}.

\section{Orientations and bidirections in graphs}
Let $\Gamma=(V_\Gamma,E_\Gamma)$ be a finite, connected, simple, undirected graph, with at least one edge. The set $V_\Gamma$ is the vertex set, and the set $E_\Gamma$ is the edge set, consisting of unordered pairs of type $\{u,v\}$, with $u,v\in V_\Gamma
$. We write $u\sim v$ if $\{u,v\}\in E_\Gamma$, then  we say that $u$ and $v$ are \emph{adjacent} and that are \emph{endpoints} of the edge $\{u,v\}$.
We will use the set theoretic notation for the edges: for $v\in V_\Gamma$ and $e\in E_\Gamma$ we write $v\in e$ if the edge $e$ is \emph{incident} to $v$, that is, $v$ is one of the endpoints of $e$. If $e_1,e_2\in E_\Gamma$ are both incident to a vertex, we denote that vertex as $e_1\cap e_2$. We write $e_1\cap e_2=\emptyset$ if $e_1$ and $e_2$ do not share a common vertex. For any $v\in V_\Gamma$, we denote by $\deg(v)$ the \emph{degree} of $v$, that is, the number of edges that are incident to $v$. A \emph{walk} $W$ of length $k-1$ in $\Gamma$ is an ordered list of $k$ vertices $v_{1},\ldots, v_{k}$ such that $v_i\sim v_{i+1}$. The walk $W$ is closed if $v_1=v_k$.

Throughout this paper $|V_\Gamma|=n$ and $|E_\Gamma|=m$.
We fix an order  $V_\Gamma=\{v_1,\ldots,  v_n\}$ and we write the \emph{adjacency matrix} $A_{\Gamma}$ of $\Gamma$ with respect to this order as:
$$
(A_{\Gamma})_{i,j}:=
 \begin{cases} 1 &\mbox{if } v_i\sim v_j\\
 0 &\mbox{otherwise.}
 \end{cases}
$$
We also denote by $\Delta_{\Gamma}$ and $\Delta_{\Gamma}^+$ the \emph{Laplacian} and the \emph{signless Laplacian} of $\Gamma$, respectively:
\begin{equation*}
\begin{split}
\Delta_{\Gamma}&:= \deg(\Gamma)- A_{\Gamma}\\
\Delta_{\Gamma}^{+}&:= \deg(\Gamma)+A_{\Gamma},
\end{split}
\end{equation*}
where the matrix $\deg(\Gamma)$ is the diagonal matrix of the degrees of the vertices $V_\Gamma$.

The \emph{line graph} of $\Gamma$ is the graph $L(\Gamma)=(V_{L(\Gamma)}, E_{L(\Gamma)})$ where $V_{L(\Gamma)}=E_\Gamma$, and for $e_1,e_2\in V_{L(\Gamma)}$ we have $e_1\sim e_2$ if $e_1\cap e_2\neq\emptyset$. Fixing also an order of $E_\Gamma=\{e_1,\ldots,e_m\}$ we can define the  adjacency matrix $A_{L(\Gamma)}$ of the line graph and the \emph{incidence matrix} $N$ of $\Gamma$, the latter being the $n\times m$ matrix such that
\begin{eqnarray}\label{eq:N}
 N_{i,j}:=
 \begin{cases} 1 &\mbox{if } v_i\in e_j\\
 0 &\mbox{otherwise.}
 \end{cases}
\end{eqnarray}
We will denote by $O(E_\Gamma)$ the set of \emph{oriented edges}, that is, the set of the ordered pairs consisting of adjacent vertices
$O(E_\Gamma) =\{(u,v): u,v\in V_\Gamma, u\sim v\}$. An oriented edge  $(u,v)\in O(E_\Gamma)$ is thought to go from $u$ to $v$. By definition, there exists a 2-1-projection $p\colon O(E_\Gamma)\to E_\Gamma$ such that $p((u,v))=p((v,u))=\{u,v\}$.\\
\indent An \emph{orientation} $\mathfrak o$ of $\Gamma$ is a subset of $O(E_\Gamma)$ such that the restriction of $p$ to $\mathfrak o$ is a bijection, that is, $\mathfrak o$  assigns a direction to each edge.
If $e=\{v_1,v_2\} \in E_{\Gamma}$, and $(v_1,v_2)\in \mathfrak o$, we denote with $e^{\mathfrak o}:=(v_1,v_2)$ the associated oriented edge.
An example of orientation $\mathfrak o_<$ is the one induced by the order of $V_\Gamma$,  that is $\mathfrak o_<:=\{(v_i,v_j): \{v_i,v_j\}\in E_\Gamma, i<j \}$.

With the graph $\Gamma$ one can also associate a \emph{bidirection} $\mathfrak h$, that is a choice, for every edge, of an independent orientation
at each of its endpoints (see \cite{bid,zasmat}). The bidirection $\mathfrak h$ can be formalized as a map
$\mathfrak h\colon \{(v,e)\in V_\Gamma\times E_\Gamma: v\in e \}\to \{\pm 1\}$, with the interpretation that if $\mathfrak h(v,e)=1$ then the  arrow of the edge $e$ near $v$  points to $v$, if $\mathfrak h(v,e)=-1$ then the arrow of the edge $e$ near $v$ goes away from $v$ (see \cite{zasmat}). Moreover, we can also consider the \emph{incidence matrix $N_{\mathfrak h}$ associated with the bidirection $\mathfrak h$}: more precisely, $N_{\mathfrak h}$ is the $\{0,\pm1\}$-valued $n\times m$ matrix such that
\begin{equation}\label{eq:bidi}
(N_{\mathfrak h})_{i,j}:=
\begin{cases}
\mathfrak h(v_i,e_j) &\mbox{if } v_i\in e_j\\
 0 &\mbox{otherwise.}
\end{cases}
\end{equation}
Notice that the positions of the nonzero entries of $N_{\mathfrak h}$ are the same of the nonzero entries of $N$, for every bidirection $\mathfrak h$.

\section{Gain graphs and group algebra valued matrices}\label{section4}
Let $G$ be a group. We consider a map $\psi\colon O(E_\Gamma)\to G$ such that $\psi(u,v)=\psi(v,u)^{-1}$.
The pair $(\Gamma,\psi)$ is a \emph{$G$-gain graph} (or equivalently, a gain graph on $G$) and $\psi$ is said to be a \emph{gain function} (or $G$-gain function), and $\Gamma$ is said to be the \emph{underlying graph} of the gain graph $(\Gamma,\psi)$.\\
\indent We denote by $G(\Gamma)$ the set of all possible $G$-gain functions on $\Gamma$.
For a given gain function $\psi$ and a walk $W=v_1,\ldots,v_k$ in $\Gamma$, the gain of $W$ is defined as $\psi(W):=\psi(v_1,v_2)\psi(v_2,v_3)\cdots\psi(v_{k-1},v_k)$. A gain graph $(\Gamma,\psi)$ is said to be \emph{balanced} if $\psi(W)=1_G$ for every closed walk $W$, where  $1_G$ denotes the neutral element of $G$.
A fundamental concept in the theory of gain graphs, inherited from the theory of signed graphs, is the \emph{switching equivalence}.

\begin{definition}\label{defswe}
Two gain functions $\psi_1$ and $\psi_2$ on the same underlying graph
$\Gamma$ are switching equivalent, and we shortly write $\psi_1\sim \psi_2$, if there exists $f\colon V_\Gamma\to G$ such that
\begin{equation}\label{eqsw}
\psi_2(v_i,v_j)=f(v_i)^{-1} \psi_1(v_i,v_j)f(v_j), \qquad \forall v_i,v_j: v_i\sim v_j.
\end{equation}
We write $\psi_2=\psi_1^f$ when Eq. \eqref{eqsw} holds. We denote by $[G(\Gamma)]$ the set of the switching equivalence classes of $G$-gain functions on $\Gamma$.
\end{definition}

An element $s\in G$ such that $s^2=1_G$ is called a \emph{weak involution} (see \cite{reff0}), and we denote by $\bold{s}\in G(\Gamma)$ the constant gain function such that $\bold{s}(u,v)=s$ for any $u,v\in V_\Gamma$, with $u\sim v$. Notice that, since $s=s^{-1}$, the map $\bold{s}$ is actually a gain function.
It turns out that $(\Gamma,\psi)$ is balanced if and only if $\psi\sim \bold{1_G}$ (see \cite[Lemma~5.3]{zaslavsky1}). Similarly, if $G$ is a subgroup of $\mathbb C$ containing $-1$, one says that a $G$-gain graph $(\Gamma,\psi)$ is \emph{antibalanced} if $\psi\sim \bold{-1}$.

We denote by $Z(G)=\{g\in G: gh=hg, \; \forall h\in G\}$ the \emph{centrum} of $G$ and we say that $g\in G$ is \emph{central} if $g\in Z(G)$.\\
\indent Associated with the group $G$, we consider the group algebra $\mathbb C G$ of finite $\mathbb C$-linear combinations of elements of $G$. An element $f\in \mathbb C G$ can be expressed as $f=\sum_{x\in G} f_x x$, with $f_x\in \mathbb C$, where we assume that the set $\{x\in G: f_x\neq 0\}$ is finite. The product in $\mathbb C G$ is the linear extension of the one in $G$:
$$
\left(\sum_{x\in G} f_x x\right)\cdot \left(\sum_{y\in G} h_y y\right):= \sum_{x,y\in G} f_x h_y \, x y\qquad \mbox{ for each } f,h\in \mathbb CG.
$$
Moreover, there is also an involution $^*$ on $\mathbb CG$ defined as $f^*:=\sum_{x\in G} \overline{f_x} x^{-1}$, where $\overline{f_x}$ denotes the complex conjugate of $f_x$. Clearly, there exists an embedding of $G$ in $\mathbb C G$ and we simply write $0$ to indicate the zero-vector of $\mathbb C G$.

Consider now the $\mathbb C$-vector space $M_{n\times m}(\mathbb C G)$ consisting of the $n\times m$ matrices with entries in $\mathbb C G$.
For a given $A\in M_{n\times m}(\mathbb C G)$, we define the matrix $A^*\in M_{m\times n}(\mathbb C G)$ such that
$$
(A^*)_{i,j}=(A_{j,i})^*,
$$
where the $^*$ on the right is the involution of $\mathbb C G$.
As in the case of classical matrices, it is possible to define the product of an element $A\in M_{n\times m}(\mathbb C G)$ and an element $B\in M_{m\times q}(\mathbb C G)$ in the following way:
$$
(AB)_{i,j}=\sum_{l=1}^m A_{i,l} B_{l,j},
$$
where the product on the right is that of $\mathbb C G$. In particular, the space $M_{n\times n}(\mathbb C G)$, or simply $M_{n}(\mathbb C G)$, is also an algebra with involution $^*$. For more details see \cite{nostro}. \\
\indent An element $F\in M_{n}(\mathbb C G)$ is said to be \emph{diagonal} if $F_{i,j}=0\in \mathbb C G$ for each $i\neq j$. In this case, we use the notation $F = diag(f_1,\ldots, f_n)$, with $f_i=F_{i,i}$. A matrix $A\in M_{n\times m}(\mathbb C G)$ can also be multiplied, on the left or on the right, by an element $a\in \mathbb C G$, in the following way:
\begin{eqnarray}\label{ettore}
(aA)_{i,j}=aA_{i,j}  \qquad  (Aa)_{i,j}=A_{i,j}a.
\end{eqnarray}
Notice that $aA=diag(\underbrace{a,a,\ldots,a}_{n \textrm{ times}})\, A$ and $Aa=A\, diag(\underbrace{a,a,\ldots,a}_{m \textrm{ times}})$.

The construction of the algebra $M_n(\mathbb C G)$ allows us to define adjacency and Laplacian matrices of a $G$-gain graph even when the group $G$ is not embedded into $\mathbb C$.

\begin{definition}
Let $V_\Gamma = \{v_1,v_2,\ldots, v_n\}$ be an order for the vertex set of $\Gamma$ and let $s\in Z(G)$ be such that $s^2=1_G$. The \emph{adjacency matrix} $A_{\Gamma,\psi}\in M_n(\mathbb C G)$ of $(\Gamma,\psi)$ is the matrix defined by
$$
(A_{\Gamma,\psi})_{i,j}=
 \begin{cases}
 \psi(v_i,v_j) &\mbox{if } \{v_i,v_j\}\in E_\Gamma
  \\  0 &\mbox{otherwise.}
\end{cases}
$$
The \emph{$s$-Laplacian matrix} $\Delta_{\Gamma, \psi}^s\in M_n(\mathbb C G)$ of $(\Gamma,\psi)$ is the matrix defined by
$$
\Delta_{\Gamma, \psi}^s:=\deg(\Gamma, G) + s A_{\Gamma, \psi},
$$
where $\deg(\Gamma,G)=diag(\deg(v_1) 1_G,\ldots,\deg(v_n) 1_G) \in M_n(\mathbb C G)$.
\end{definition}

\begin{remark}\label{rem:u}
If an order of $V_\Gamma$ is fixed, two gain functions $\psi_1$ and $\psi_2$ are equal if and only if $A_{\Gamma, \psi_1}^s=A_{\Gamma, \psi_2}^s$ or, equivalently, if and only if $\Delta_{\Gamma, \psi_1}^s=\Delta_{\Gamma, \psi_2}^s$.
\end{remark}

We are going to introduce a suitable equivalence relation $\sim$ in $M_n(\mathbb C G)$ in order to characterize the switching equivalence of gain functions in terms of equivalence of adjacency or Laplacian matrices (Theorem \ref{th:1}).

\begin{definition}
Let $A,B\in M_n(\mathbb C G)$. Then $A\sim B$ if there exists a diagonal matrix $F\in M_n(\mathbb C G)$ such that
$F_{i,i}\in G$, for each $i=1,\ldots, n$, and $F^*AF=B$. If this is the case, we write $B=A^F$.
\end{definition}

The following is a generalization of \cite[Theorem~1]{nostro}. For the sake of completeness, we present the proof in full.
\begin{theorem}\label{th:1}
Let $\psi_1,\psi_2\in G(\Gamma)$. The following are equivalent:
\begin{enumerate}
\item $\psi_1\sim\psi_2$;
\item $A_{\Gamma,\psi_1}\sim A_{\Gamma,\psi_2}$;
\item $\Delta_{\Gamma, \psi_1}^s\sim \Delta_{\Gamma,\psi_2}^s$.
\end{enumerate}
\end{theorem}
\begin{proof}${}$\\
(1)$\iff$(2)\\
If $\psi_1\sim \psi_2$, then there exists a map $f\colon V_\Gamma\to G$ such that Eq. \eqref{eqsw} holds.
The diagonal matrix $F\in M_n(\mathbb C G)$ with entries $F_{i,i}=f(v_i)$ satisfies $F^*A_{\Gamma,\psi_1}F=A_{\Gamma,\psi_2}$, since:
\begin{eqnarray*}
(F^*A_{\Gamma,\psi_1}F)_{i,j}&=&\sum_{r=1}^n\sum_{p=1}^n (F^*)_{i,r}{(A_{\Gamma,\psi_1})}_{r,p} F_{p,j}=\sum_{r=1}^n\sum_{p=1}^n (F_{r,i})^*{(A_{\Gamma,\psi_1})}_{r,p} F_{p,j}\\
&=& f(v_i)^{-1} {(A_{\Gamma,\psi_1})}_{i,j} f(v_j)=
\begin{cases}
f(v_i)^{-1} \psi_1(v_i,v_j)f(v_j) &\mbox{if } \{v_i,v_j\}\in E_\Gamma\\
0 &\mbox{otherwise,}
\end{cases}
\end{eqnarray*}
which corresponds exactly to the entry ${(A_{\Gamma,\psi_2})}_{i,j}$ by Eq. \eqref{eqsw}.
Vice versa, if there exists a diagonal matrix $F\in M_n(\mathbb C G)$ such that $F^*A_{\Gamma,\psi_1}F=A_{\Gamma,\psi_2}$, we can define $f(v_i):=F_{i,i}$ and easily verify that Eq. \eqref{eqsw} holds.\\
(2)$\iff$(3)\\
Clearly
$$
\deg(\Gamma, G)^F=F^*\deg(\Gamma, G) F=\deg(\Gamma, G)
$$
for any diagonal $F\in  M_n(\mathbb C G)$ such that $F_{i,i}\in G$. It follows, by using the fact that $s$ is a central element of $G$, that $A_{\Gamma,\psi_2}=\left(A_{\Gamma,\psi_1}\right)^F$ if and only if $\Delta_{\Gamma, \psi_2}^s=\left(\Delta_{\Gamma,\psi_1}^s\right)^F$.
\end{proof}

\section{$G$-phases and gain-line graphs}\label{sec:phases}
We open this core section of the paper with a key definition.
\begin{definition}\label{g-phase}
For a graph $\Gamma=(V_\Gamma, E_\Gamma)$ and a group $G$, with $|V_\Gamma|=n$ and $|E_\Gamma|=m$, a matrix $H\in M_{n\times m}(\mathbb C G)$ is said to be a \emph{$G$-phase of $\Gamma$} if
 \begin{equation}\label{cond}
\begin{split}
 &H_{i,j}\in G \mbox{ if } v_i\in e_j\quad \mbox{ and}\\
&H_{i,j}=0  \mbox{ if } v_i\notin e_j.
\end{split}
\end{equation}
We denote by  $\mathcal H_\Gamma\subseteq  M_{n\times m}(\mathbb C G) $ the set of all $G$-phases of $\Gamma$. We call  \emph{$G$-phased graph} the pair $(\Gamma,H)$, with $H\in  \mathcal H_\Gamma$ (see \cite{reff0}).
\end{definition}
\begin{example}\label{ex:N}
Let $N_\Gamma(G)\in M_{n\times m}(\mathbb C G)$ be the matrix such that
\begin{equation*}
N_\Gamma(G)_{i,j}=
\begin{cases}
1_G &\mbox{ if } v_i\in e_j\\
0 &\mbox{ if } v_i\notin e_j.
\end{cases}
\end{equation*}
Clearly $N_\Gamma(G)\in \mathcal H_\Gamma$.
\end{example}

\begin{example}
When $G=\{\pm1\}$, the set $\mathcal H_\Gamma$  consists of all possible $2^{2m}$ matrices obtained from the incidence matrix $N$ (see Eq. \eqref{eq:N}) by replacing some $1$ with $-1$.
That is, $\mathcal H_\Gamma$ is in bijection with the set of all possible bidirections of $\Gamma$ (see Eq.~\eqref{eq:bidi}).
\end{example}

Let us denote by $G^n$ the $n$-direct power of the group $G$, so that $G^n =\{ g=(g_1,\ldots,  g_n) : g_i\in G\}$, where the group product is defined coordinate-wise.
Notice that, when an order of $V_\Gamma$ is fixed, an element $g\in G^n$ can be regarded as a map $g\colon V_\Gamma\to G$ such that $g(v_i)=g_i$.

Moreover, the group $G^n$ is canonically isomorphic to the group of all diagonal matrices of $M_{n}(\mathbb C G)$ with diagonal entries in $G$, via the bijection
$$
g \mapsto \underline{g}:=diag(g_1,\ldots,  g_n)=\begin{pmatrix}g_{1} & & \\ & \ddots & \\ & & g_{n}\end{pmatrix} \in M_{n}(\mathbb C G).
$$
In particular, we can multiply a matrix $M_{n\times m}(\mathbb C G)$ by an element of $G^n$ on the left or by an element of $G^m$ on the right. These actions can be restricted to $\mathcal H_\Gamma $, as we show in the next lemma.

\begin{lemma}\label{lem:pippo}
For any $f\in G^n, g\in G^m,  H\in \mathcal H_\Gamma$ one has $\underline{f}H\underline{g}\in \mathcal H_\Gamma$.
\end{lemma}
\begin{proof}
Since $(\underline{f}H\underline{g})_{i,j}=f_i H_{i,j} g_j$  then clearly if $H$ satisfies conditions in Eq. \eqref{cond}  then  also $\underline{f}H\underline{g}$ does.
\end{proof}

An action $P$ of a group $G$ over a set $S$ is a map
$$
P\colon G\times S\to S
$$
such that the induced map
$$
\mathcal P \colon G\to Sym(S),\qquad g\mapsto P_g
$$
is a group homomorphism, where $P_g(s)=P(g,s)$, for every $g\in G$ and $s\in S$, and $Sym(S)$ denotes the permutation group of $S$. For an element $s\in S$, we denote with $St_P(s)=\{g\in G: P(g,s)=s\}$ the \emph{stabilizer} of $s$. Notice that $St_P(s)$ is a subgroup of $G$. We say that the action $P$ is \emph{free} if the stabilizer of each element is trivial or, equivalently, if $P_g$ has no fixed points, for every $g\in G$.
A \emph{$P$-orbit} $\mathcal O$ of the action $P$ is a subset of $S$ for which there exists $s\in S$ such that $\mathcal O=\{P(g,s)|g\in G\}$. The orbits of $P$ define a partition of $S$, whose associated equivalence relation will be denoted by $\sim_P$.

\begin{definition}
Let $\Gamma$ be a graph. Then three actions on the set $\mathcal H_\Gamma$ of $G$-phases of $\Gamma$ can be defined as follows. For every $ f\in G^n, g\in G^m, H\in  \mathcal H_\Gamma,$ we put:
\begin{align*}
r&\colon  G^m\times  \mathcal H_\Gamma \to \mathcal H_\Gamma, \qquad
&r(g,H)&=H\underline{g}\\
l&\colon  G^n\times  \mathcal H_\Gamma \to \mathcal H_\Gamma, \qquad
&l(f,H)&= \underline{f}^*H\\
(l\times r) &\colon  (G^{n}\times G^m)\times  \mathcal H_\Gamma \to \mathcal H_\Gamma, \qquad
&(l\times r)((f,g),H)&=\underline{f}^*H\underline{g}.
\end{align*}
\end{definition}

The actions $r$ and $l$ are clearly free. The stabilizer $St_{l\times r}(H)$, for $H\in \mathcal H_\Gamma$, is in general nontrivial and depends on $H$. The abelian case is analyzed in the following proposition.

\begin{proposition}\label{prop:st}
Let $G$ be an abelian group. Then, for any $H\in \mathcal{H}_\Gamma$:
$$
St_{l\times r}(H)=\{(f,g)\in G^{n}\times G^m: f_1=\cdots=f_n=g_1=\cdots=g_m\}\cong G.
$$
\end{proposition}
\begin{proof}
Clearly if $(f,g)\in G^{n}\times G^m$ is such that  $f_1=\cdots=f_n=g_1=\cdots=g_m$ then, for every $H\in H_\Gamma$, we have
\begin{equation}\label{eq:comm}
(l\times r)((f,g),H)_{i,j}=f_i^{-1}H_{i,j}g_j=H_{i,j}.
\end{equation}
Vice versa, if Eq. \eqref{eq:comm} holds and $G$ is abelian, for every $v_i\in e_j$ we have $f_i=g_j$. The thesis follows by connectedness of $\Gamma$.
\end{proof}
We write $H_1\sim_r H_2$ if $H_1,H_2\in  \mathcal H_\Gamma$  are in the same $r$-orbit, that is, if there exists $g\in G^m$ such that $H_1=H_2 \underline{g}$.
If this is the case, one has:
\begin{equation}\label{eq:hhs}
H_1H_1^*= H_2\underline{g}(H_2\underline{g})^*=   H_2\underline{g}\underline{g}^*H_2^*=H_2 H_2^*.
\end{equation}
Similarly, we write $H_1\sim_l H_2$ if $H_1,H_2\in  \mathcal H_\Gamma$  are in the same $l$-orbit, that is, there exists $f\in G^n$ such that $H_1=  \underline{f}^*H_2 $.
If this is the case, one has:
\begin{equation}\label{eq:hsh}
H_1^*H_1= (\underline{f}^*H_2)^*  \underline{f}^*H_2=H_2^* \underline{f}   \underline{f}^*H_2=H_2^*H_2.
\end{equation}
Finally, we write $H_1\sim_{l\times r} H_2$ if $H_1,H_2\in  \mathcal H_\Gamma$  are in the same $(l\times r)$-orbit,  that is, there exists $f\in G^n$   and  $g\in G^m$ such that
$H_1=  \underline{f}^*H_2 g$.
\begin{remark}\label{rem:transitive}
For any $H_1,H_2,H_3\in \mathcal H_\Gamma$, one has:
$$
H_1\sim_r H_2 \mbox{ and } H_2\sim_l H_3 \implies H_1\sim_{l\times r} H_3.
$$
\end{remark}
We conclude by defining a further equivalence relation, denoted by $\sim_{l\cap r}$, whose associated partition is the one obtained intersecting $r$-orbits with $l$-orbits:
$$
H_1\sim_{l\cap r} H_2 \iff H_1\sim_{l} H_2 \mbox{ and } H_1\sim_{r} H_2.
$$

\subsection{$G$-phases and gains of $\Gamma$}\label{sec:gamma}
The aim of this section is to associate a gain with a given graph $\Gamma$, starting from a $G$-phase of $\Gamma$. We start by giving the following definition.

\begin{definition}\label{eq:psi}
Let $s_1\in Z(G)$  with $s_1^{2}=1_{G}$, and define the map
$$
\Psi\colon \mathcal H_\Gamma \to G(\Gamma)
$$
such that, for every $H\in  \mathcal H_\Gamma$, the gain function $\Psi(H)$ on $\Gamma$ is:
\begin{equation}\label{eq:H}
\Psi(H)(v_i,v_j)=s_1 H_{i,k} (H_{j,k})^{-1}, \quad \mbox{ for } e_k=\{v_i,v_j\}.
\end{equation}
\end{definition}
Observe that $\Psi(H)$ is in fact a gain function, since
$$
\Psi(H)(v_j,v_i)=s_1 H_{j,k} (H_{i,k})^{-1}=(s_1 H_{i,k} (H_{j,k})^{-1})^{-1}=\left(\Psi(H)(v_i,v_j)\right)^{-1}.
$$
The following lemma is a generalization of Eq. \eqref{eq:1}.

\begin{lemma}\label{lem:r}
For any $G$-phase $H\in \mathcal H_\Gamma$, one has:
\begin{equation}\label{eq:r}
HH^*=\Delta^{s_1}_{\Gamma, \Psi(H)}.
\end{equation}
\end{lemma}
\begin{proof}
First of all, we have:
\begin{eqnarray*}
(HH^*)_{i,i}=\sum_{\ell=1}^m H_{i,\ell} H^*_{\ell,i}=\sum_{\ell=1}^m H_{i,\ell} H_{i,\ell}^{-1}=|\{e\in E_\Gamma: v_i\in e\}| 1_G =\deg(v_i) 1_G.
\end{eqnarray*}
Now, for $i\neq j$ and $e_k=\{v_i,v_j\}$:
\begin{eqnarray*}
(HH^*)_{i,j}=\sum_{\ell=1}^m H_{i,\ell} H^*_{\ell,j}=\sum_{\ell=1}^m H_{i,\ell} H_{j,\ell}^{-1}=H_{i,k} H_{j,k}^{-1}={s_1}\Psi(H)(v_i,v_j).
\end{eqnarray*}
Finally, if $i\neq j$ and $v_i \nsim v_j$:
\begin{eqnarray*}
(HH^*)_{i,j}=\sum_{\ell=1}^m H_{i,\ell} H^*_{\ell,j}=\sum_{\ell=1}^m H_{i,\ell} H_{j,\ell}^{-1}=0
\end{eqnarray*}
and the proof is completed.
\end{proof}

\begin{remark}
Our Lemma \ref{lem:r} generalizes the content of \cite[Eq.~IV.1]{zasmat} when $G=\{\pm 1\}$ and $s_1=-1$, and the content of \cite[Proposition~4]{francesco}, when $G=\mathbb{T}_4=\{1, i, -1, -i\}$ and $s_1=-1$ (see also Remark~\ref{rem:ponte} of Section \ref{sec:rep}).
\end{remark}

\noindent  For a fixed $G$-gain function $\psi$ of $\Gamma$, one can consider the sets
\begin{equation*}\begin{split}
\Psi^{-1}(\psi)&=\{H\in \mathcal H_\Gamma: \Psi(H)=\psi\},\\
\Psi^{-1}([\psi])&=\{H\in \mathcal H_\Gamma: \Psi(H)\sim\psi\}.
\end{split}
\end{equation*}
The set $\Psi^{-1}(\psi)$ will be called the \emph{fiber of $\psi$}, and it is never empty, for any $\psi$, as we will show in Proposition \ref{prop:sur}.

In \cite{reff0}, in analogy with \cite{zaslavskyori,zasmat}, a triple $(\Gamma,\psi,H)$, with $H\in \mathcal H_\Gamma$ such that $\Psi(H)=\psi\in G(\Gamma)$, is called an \emph{oriented $G$-gain graph}. We are going to compare this notion with the classical orientations of $\Gamma$.

\begin{definition}\label{H0}
Let $\mathfrak o$ be an orientation of $\Gamma$. Let us define the map
$$
H_{\mathfrak o}\colon G(\Gamma)\to \mathcal H_\Gamma
$$
such that, for every gain function $\psi$ on $\Gamma$:
\begin{equation*}
H_{\mathfrak o}(\psi)_{i,k}:=
\begin{cases}
0 & \mbox{ if } v_i\not\in e_k\\
\psi(v_i,v_j) &\mbox{ if } e^{\mathfrak o}_k=(v_i,v_j)\\
s_1 &\mbox{ if } e^{\mathfrak o}_k=(v_j,v_i).
\end{cases}        \
\end{equation*}
\end{definition}

\begin{example}\label{ex:paw}
Consider the Paw Graph $P=(V_P,E_P)$ with a fixed order for the vertices $V_P$, a fixed order for the edges $E_P$ and the orientation $\mathfrak o_<$, depicted on the left of Fig. \ref{fig:paw}.
Consider the Quaternion group $\mathbb{Q}_8=\{\pm 1,\pm i,\pm j,\pm k\}$, with the fixed central involution $s_1=-1$. Let $(P,\psi)$ be the $\mathbb{Q}_8$-gain graph whose adjacency matrix is
$$
A_{P,\psi}= \begin{pmatrix}
0& -i& 0& 0\\
i & 0& -j & -i\\
0&j&0 & -k \\
0&i &k&0
\end{pmatrix}.
$$
Notice that the standard way to draw a gain graph, highlighting the gain of only one orientation for each edge, already contains the choice of an orientation. On the right of Fig. \ref{fig:paw} we have chosen it coherently with $\mathfrak o_<$.

\begin{figure}
\begin{center}
\begin{tikzpicture}[scale=0.75]

\begin{scope}
[decoration={
    markings,
    mark=at position 0.5 with {\arrow{Latex[length=3.4mm, width=1.2mm]}}}
    ]
\tikzset{vertex/.style = {draw,circle,fill=black,minimum size=4pt,
                            inner sep=0pt}}
\tikzset{uedge/.style = {-,> = latex'}}

\node[vertex] (v3) at  (0,0) [scale=0.8,label=left:$v_3$]{};
\node[vertex] (v4) at  (3,0) [scale=0.8,label=right:$v_4$]{};
\node[vertex] (v2) at  (1.5,2.5) [scale=0.8,label=left:$v_2$]{};
\node[vertex] (v1) at  (1.5,5) [scale=0.8,label=left:$v_1$]{};

\node[vertex] (w3) at  (0+6,0) [scale=0.8,label=left:$v_3$]{};
\node[vertex] (w4) at  (3+6,0) [scale=0.8,label=right:$v_4$]{};
\node[vertex] (w2) at  (1.5+6,2.5) [scale=0.8,label=left:$v_2$]{};
\node[vertex] (w1) at  (1.5+6,5) [scale=0.8,label=left:$v_1$]{};

\draw[postaction={decorate}]  (v1) -- (v2)node[scale=0.8,midway,left] {$e^{\mathfrak o_<}_1$};
\draw[postaction={decorate}]  (v2) -- (v3)node[scale=0.8,midway,left] {$e^{\mathfrak o_<}_2$};
\draw[postaction={decorate}]  (v3) -- (v4)node[scale=0.8,midway,below] {$e^{\mathfrak o_<}_3$};
\draw[postaction={decorate}] (v2) -- (v4)node[scale=0.8,midway,right] {$e^{\mathfrak o_<}_4$};

\draw[postaction={decorate}]  (w1) -- (w2)node[scale=0.8,midway,left] {$-i$};
\draw[postaction={decorate}]  (w2) -- (w3)node[scale=0.8,midway,left] {$-j$};
\draw[postaction={decorate}]  (w3) -- (w4)node[scale=0.8,midway, below] {$-k$};
\draw[postaction={decorate}] (w2) -- (w4)node[scale=0.8,midway, right] {$-i$};

\end{scope}
\end{tikzpicture}
\end{center}\caption{The Paw graph $P$ with orientation $\mathfrak o_<$ and the gain graph $(P,\psi)$.}\label{fig:paw}
\end{figure}
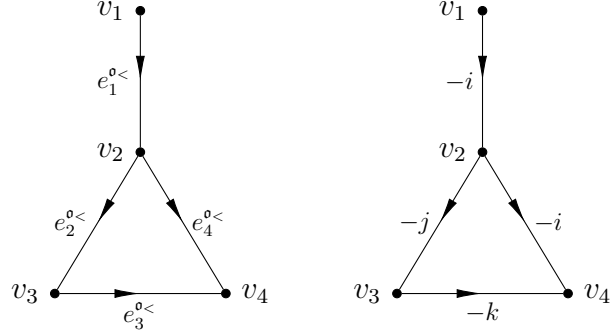
Following Definition \ref{H0} we have:
$$
H_{\mathfrak o}(\psi)= \begin{pmatrix}
-i & 0& 0& 0 \\
-1 & -j& 0 & -i\\
0 &-1 &-k & 0 \\
0&0 &-1&-1
\end{pmatrix}.
$$
\end{example}

\begin{proposition}\label{prop:sur}
For every orientation $\mathfrak o$ of $\Gamma$, the map $H_{\mathfrak o}$ is a \emph{section} of the map $\Psi\colon \mathcal{H}_\Gamma\to G(\Gamma)$, that is, for every $\psi\in G(\Gamma)$ we have $\Psi(H_{\mathfrak o}(\psi))=\psi$. In particular, every fiber $\Psi^{-1}(\psi)$ is not empty.
\end{proposition}
\begin{proof}
Let $\psi\in G(\Gamma)$. Then, by using Eq. \eqref{eq:H}, we have:
\begin{equation*}\begin{split}
\Psi(H_{\mathfrak o}(\psi))(v_i,v_j)&=
\begin{cases}
0 &\mbox{ if } v_i\nsim v_j\\
s_1H_{\mathfrak o}(\psi)_{i,k} (H_{\mathfrak o}(\psi)_{j,k})^{-1} &\mbox{ if } e_k = \{v_i,v_j\}
\end{cases}\\
&=\begin{cases}
0 &\mbox{ if } v_i\nsim v_j\\
s_1 \psi(v_i,v_j) s_1^{-1} &\mbox{ if } (v_i,v_j)\in \mathfrak o\\
s_1  s_1 \left( \psi(v_j,v_i)\right)^{-1} &\mbox{ if } (v_j,v_i)\in \mathfrak o
\end{cases}\\&=\psi(v_i,v_j),
\end{split}
\end{equation*}
and so $H_{\mathfrak o}(\psi)\in \Psi^{-1}(\psi)$.
\end{proof}

\begin{theorem}\label{th:r}
The orbits of the action $r$ on $\mathcal H_\Gamma$ are exactly the fibers of $\Psi$. In other words, for $H_1,H_2\in\mathcal H_\Gamma$, one has that $H_1\sim_r H_2$
if and only if $\Psi(H_1)=\Psi(H_2)$.
\end{theorem}
\begin{proof}
Suppose $H_1\sim_r H_2$. It follows, by using Eq. \eqref{eq:hhs} and Lemma \ref{lem:r}, that
\begin{equation*}
\Delta^{s_1}_{\Gamma,\Psi(H_1)}=H_1H_1^*=H_2H_2^*=\Delta^{s_1}_{\Gamma,\Psi(H_2)}.
\end{equation*}
Combining with Remark \ref{rem:u} we obtain $\Psi(H_1)=\Psi(H_2)$.

Let us prove the converse implication. We start by proving that, for any orientation $\mathfrak o$ and any gain $\psi$, one has:
\begin{equation}\label{eq:exist}
H\in \Psi^{-1}(\psi)\implies H=r(g,H_{\mathfrak o}(\psi))=H_{\mathfrak o}(\psi)\underline{g}
\end{equation}
for some $g\in G^m$, where the matrix $H_{\mathfrak o}(\psi)$ is the one of Definition~\ref{H0}.
In order to do that, for any $H\in \Psi^{-1}(\psi)$, for each $k=1,\ldots, m$, with $e_k=\{v_i,v_j\}$ and $(v_i,v_j)\in \mathfrak o$, we define $g_k:={s_1} H_{j,k}$. Now we check that, if we choose $g=(g_1,\ldots,g_m)$, then Eq. \eqref{eq:exist} holds. More precisely:
\begin{equation}\label{eq:vice}
\left(H_{\mathfrak o}(\psi)\underline{g}\right)_{i,k}=H_{\mathfrak o}(\psi)_{i,k}g_k=
\begin{cases}
 0 &\mbox{ if } v_i\notin e_k\\
\psi(v_i,v_j){s_1} H_{j,k} &\mbox{ if } e_k=\{v_i,v_j\},\, (v_i,v_j) \in \mathfrak o\\
{s_1} {s_1} H_{i,k}&\mbox{ if } e_k=\{v_i,v_j\},\, (v_j,v_i) \in \mathfrak o.
\end{cases}
\end{equation}
Notice that, since $H\in \Psi^{-1}(\psi)$, it must be $\Psi(H)=\psi$; therefore, for $e_k=\{v_i,v_j\}$, Eq. \eqref{eq:H} gives
$$
\psi(v_i,v_j)={s_1} H_{i,k} H_{j,k}^{-1}.
$$
Combining with Eq. \eqref{eq:vice}, we have
\begin{equation*}
\begin{split}
\left(H_{\mathfrak o}(\psi)\underline{g}\right)_{i,k}&=\begin{cases}
 0 &\mbox{ if } v_i\notin e_k\\
 H_{i,k} &\mbox{ if } e_k=\{v_i,v_j\},\, (v_i,v_j) \in \mathfrak o\\
 H_{i,k}&\mbox{ if } e_k=\{v_i,v_j\},\, (v_j,v_i) \in \mathfrak o
\end{cases}
\\&= H_{i,k}
\end{split}
\end{equation*}
and then $H\sim_r H_{\mathfrak o}(\psi)$. Now let $\psi:=\Psi(H_1)=\Psi(H_2)$. Then by Eq. \eqref{eq:exist} we have $H_1\sim_r H_{\mathfrak o}(\psi)$ and  $H_2\sim_r H_{\mathfrak o}(\psi)$.
Then $H_1\sim_r H_2$ by transitivity.
\end{proof}

In particular, Theorem \ref{th:r} shows that, for any $H\in \mathcal H_\Gamma$ and any orientation $\mathfrak o$, one has $H \sim_r H_{\mathfrak o}(\Psi(H))$. As a consequence, every $G$-phase is $\sim_r$ equivalent to a $G$-phase with at least one entry equal to $s_1$ in each column.

\begin{remark}
The novelty with respect to the existing literature (e.g. \cite{francesco, reff0,reff1,zasmat}) is that, even when $G<\mathbb T$, our matrices take values in the group algebra and not in the complex field. Anyway, as we will show in Section~\ref{sec:rep} (see Remark~\ref{rem:ponte}), we can obtain complex matrices by ``representing'' our abstract $\mathbb C G$-valued matrices. In this sense, we can say that our Theorem \ref{th:r} generalizes the content of \cite[Proposition~3]{francesco} (when $G=\mathbb T$ and $s_1=-1$).
\end{remark}

\subsection{$G$-phases and gains of $L(\Gamma)$}\label{sec:Lgamma}
The aim of this section is to construct a gain for the line graph $L(\Gamma)$, starting from a $G$-phase of $\Gamma$.

\begin{definition}\label{def:Lpsi}
Let $s_2\in Z(G)$ such that $s_2^{2}=1_{G}$. Let us introduce the map
$$
\Psi_L\colon \mathcal H_\Gamma \to G(L(\Gamma))
$$
such that, for every $H\in  \mathcal H_\Gamma$, the gain function $\Psi_L(H)$ on $L(\Gamma)$ is defined by:
\begin{equation}\label{eq:LH}
\Psi_L(H)(e_i,e_j)=s_2 H_{k,i}^{-1} H_{k,j}, \quad \mbox{ for } v_k= e_i \cap e_j.
\end{equation}
\end{definition}
Notice that $\Psi_L(H)$ is in fact a gain function, since
$$
\Psi_L(H)(e_j,e_i)=s_2 H_{k,j}^{-1} H_{k,i}  = (s_2 H_{k,i}^{-1} H_{k,j} )^{-1}= \left(\Psi_L(H)(e_i,e_j)\right)^{-1}.
$$
The following lemma is a generalization of Eq. \eqref{eq:2}.

\begin{lemma}\label{lem:l}
For any $G$-phase $H\in \mathcal H_\Gamma$ one has
\begin{equation}\label{eq:l}
H^*H=2\cdot \underline{ 1_{G^m} }+s_2 A_{L(\Gamma), \Psi_L(H)}.
\end{equation}
\end{lemma}
\begin{proof}
Let us start by computing the diagonal entries of the matrix $H^*H$. For each $i=1,\ldots, m$, one has:
\begin{equation*}
(H^*H)_{i,i}=\sum_{\ell=1}^n H^*_{i,\ell} H_{\ell,i}=\sum_{\ell=1}^n H_{\ell,i}^{-1} H_{\ell,i}=2\cdot1_G,
\end{equation*}
as the product $H_{\ell,i}^{-1} H_{\ell,i}$ is equal to $1_G$ for exactly two values of $\ell$, corresponding to the endpoints of the $i$-th edge, and it is $0$ otherwise.
If $i\neq j$, and $v_k=e_i \cap e_j$, then
\begin{equation*}
(H^*H)_{i,j}=\sum_{\ell=1}^n H^*_{i,\ell} H_{\ell,j}=\sum_{\ell=1}^n H_{\ell,i}^{-1} H_{\ell,j}=H_{k,i}^{-1} H_{k,j}= s_2\Psi_L(H)(e_i,e_j).
\end{equation*}
Finally, if $i\neq j$ and $e_i \cap e_j=\emptyset$, then
\begin{equation*}
(H^*H)_{i,j}=\sum_{\ell=1}^n H^*_{i,\ell} H_{\ell,j}=\sum_{\ell=1}^n H_{\ell,i}^{-1} H_{\ell,j}=0
\end{equation*}
and the proof is completed.
\end{proof}

\begin{remark}
Our Lemma \ref{lem:l} generalizes some existing results, in a sense that we will formalize in Section \ref{sec:rep} (see Remark~\ref{rem:ponte}):
\begin{itemize}
\item \cite[Eq.~V.1]{zasmat} in the case $G=\{\pm1\}$, with $s_2=-1$;
\item \cite[Theorem~1]{francesco} in the case $G=\mathbb{T}_4$, with $s_2=1$;
\item \cite[Theorem~5.1]{reff0} in the case $G=\mathbb T$, with $s_2=\pm 1$.
\end{itemize}
\end{remark}

\noindent For a fixed $G$-gain function $\zeta$ of $L(\Gamma)$, we put
\begin{equation*}\begin{split}
\Psi_L^{-1}(\zeta)&=\{H\in \mathcal H_\Gamma: \Psi_L(H)=\zeta\},\\
\Psi_L^{-1}([\zeta])&=\{H\in \mathcal H_\Gamma: \Psi_L(H)\sim\zeta\},
\end{split}
\end{equation*}
and we call $\Psi_L^{-1}(\zeta)$ the \emph{fiber of $\zeta$}.
Unlike what happens for gain functions of $\Gamma$, not every gain function of $L(\Gamma)$ can be achieved as an image of some $G$-phase via the map $\Psi_L$, as we show in Example \ref{exvirus}. However it is true that, if a switching class is ``achievable'', then every gain function of that class is (Corollary \ref{cor:nuovo}).

\begin{example}\label{exvirus}
Let us consider the star graph $S_3$ with a fixed order for the vertices and for the edges, depicted in Fig. \ref{fig:star} on the left. The line graph $L(S_3)$ is  the complete graph $K_3$ on $3$ vertices. Consider the closed walk $W= e_1,e_2,e_3,e_1$ in $L(S_3)$. Let $H \in \mathcal{H}_{S_3}$. Then:
\begin{equation*}\label{eq:tr}
\begin{split}
\Psi_L(H)(W)&=\Psi_L(H)(e_1,e_2)\Psi_L(H)(e_2,e_3)\Psi_L(H)(e_3,e_1)\\&=s_2 H_{1,1}^{-1} H_{1,2}s_2 H_{1,2}^{-1} H_{1,3}s_2 H_{1,3}^{-1} H_{1,1}=s_2.
\end{split}
\end{equation*}
Clearly, not every gain function $\psi$ on $K_3$ satisfies the property $\psi(W)=s_2$, and it follows $\Psi_L(\mathcal H_{S_3})\subsetneq G(L(S_3))$.
\begin{figure}[h]
\begin{center}
\psfrag{v1}{$v_1$}\psfrag{v2}{$v_2$}\psfrag{v3}{$v_3$}\psfrag{v4}{$v_4$}
\psfrag{e1}{$e_1$}\psfrag{e2}{$e_2$}\psfrag{e3}{$e_3$}\psfrag{e4}{$e_4$}
\includegraphics[width=0.5\textwidth]{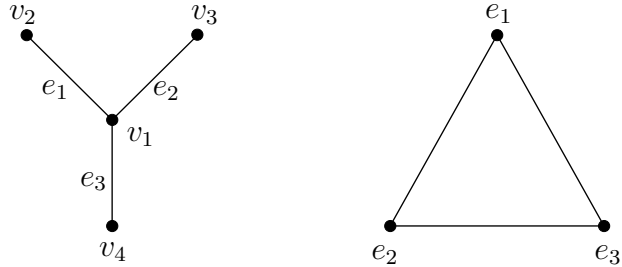}
\end{center}\caption{The Star graph $S_3$ and its line graph $L(S_3)$.}  \label{fig:star}
\end{figure}
\end{example}

\begin{definition}\label{def:compatible}
A gain function $\zeta \in G(L(\Gamma))$ is said to be a \emph{gain-line function} if $\zeta \in \Psi_L(\mathcal H_\Gamma)$.
A pair $(\psi,\zeta)\in G(\Gamma)\times G(L(\Gamma))$ is said to be \emph{compatible} if there exists $H\in \mathcal H_\Gamma$ such that
$\Psi(H)=\psi$ and $\Psi_L(H)=\zeta$.\\
\indent A gain function $\zeta$ of $L(\Gamma)$ is a gain-line function \emph{compatible} with a given $\psi\in G(\Gamma)$ if the pair $(\psi,\zeta)$ is compatible.\\
\indent Finally a gain graph $(L,\zeta)$ is said to be a \emph{gain-line graph} if there exist a graph $\Gamma$ and a $G$-phase $H\in \mathcal H_\Gamma$ such that $L(\Gamma)=L$ and $\Psi_L(H)=\zeta$.
\end{definition}

\begin{theorem}\label{th:l}
The orbits of the action $l$ on $\mathcal H_\Gamma$ are exactly the fibers of  $\Psi_L$.
In other words, for $H_1,H_2\in\mathcal H_\Gamma$, one has that $H_1\sim_l H_2$
if and only if $\Psi_L(H_1)=\Psi_L(H_2)$.
\end{theorem}
\begin{proof}
Suppose $H_1\sim_l H_2$. Then Eq. \eqref{eq:hsh} and Lemma \ref{lem:l} give
$$
A_{L(\Gamma),\Psi_L(H_1)}=A_{L(\Gamma),\Psi_L(H_2)},
$$
that, by virtue of Remark \ref{rem:u}, implies $\Psi_L(H_1)=\Psi_L(H_2)$.

Vice versa, suppose $\Psi_L(H_1)=\Psi_L(H_2)$. This implies that
\begin{equation*}
v_k= e_i\cap e_j\implies (H_1)_{k,i}^{-1} (H_1)_{k,j}=(H_2)_{k,i}^{-1} (H_2)_{k,j}
\end{equation*}
or, equivalently
\begin{equation}\label{eq:d}
v_k= e_i\cap e_j\implies (H_1)_{k,j} (H_2)^{-1}_{k,j}=(H_1)_{k,i} (H_2)^{-1}_{k,i}.
\end{equation}
For every $v_k\in V_\Gamma$, $e_i\in E_\Gamma$ such that $v_k\in e_i$, we define $f_{k,i}:=(H_1)_{k,i} (H_2)^{-1}_{k,i}\in G$.
Observe that the element $f_{k,i}$ does not depend on the particular choice of the edge containing $v_k$.
In fact, if $e_i\neq e_j$, $v_k\in e_i$ and $v_k\in e_j$, then $v_k=e_i\cap e_j$, and then, by Eq. \eqref{eq:d}, we have $f_{k,i}=f_{k,j}$.
Therefore, we  can simply write $f_{k}$ instead of $f_{k,i}$.  Let us define $f\in G^n$ as $f:=(f_1^{-1},\ldots,  f_n^{-1})$. Then we have, for $v_k\in e_i$,
$$
(l(f,H_2))_{k,i}= (\underline{f}^* H_2)_{k,i}=f_k (H_2)_{k,i}= (H_1)_{k,i} (H_2)^{-1}_{k,i}(H_2)_{k,i} = (H_1)_{k,i}.
$$
It follows that $l(f,H_2)=H_1$ and so $H_1\sim_l H_2$.
\end{proof}

In the light of Lemma \ref{lem:r} and Lemma \ref{lem:l}, the \emph{compatible pairs} in $G(\Gamma)\times G(L(\Gamma))$ are those for which a generalization of Eq. \eqref{eq:1} and Eq. \eqref{eq:2} holds. In the next section we will investigate the well posedness of our definition with respect to the switching equivalence, that is a crucial issue in the theory of gain graphs. The next corollary shows that compatible pairs are nothing but $G$-phases, up to the equivalence relation $\sim_{l\cap r}$.

\begin{corollary}\label{cor:cap}
For $H_1,H_2\in \mathcal H_\Gamma$, the following are equivalent:
\begin{enumerate}
\item $\Psi(H_1)=\Psi(H_2)$ and $\Psi_L(H_1)=\Psi_L(H_2)$;
\item $H_1\sim_{l\cap r} H_2$;
\item there exists $(f,g)\in St_{l\times r} (H_1)$ such that  $H_2=\underline{f}^* H_1$;
\item there exists $(f,g)\in St_{l\times r} (H_1)$ such that  $H_2= H_1\underline{g}$.
\end{enumerate}
\end{corollary}
\begin{proof}
The equivalence (1)$ \iff $(2) is a consequence of Theorem \ref{th:r} and Theorem \ref{th:l}.\\
(3)$ \Longrightarrow $ (2)\\
Suppose $H_2=\underline{f}^* H_1$ (and so $H_1\sim_l H_2$), and that there exists $g\in G^m$ such that $\underline{f}^* H_1\underline{g}=H_1$. This implies
$H_2 \underline{g}=H_1$, that is, $H_1\sim_r H_2$, and then $H_1\sim_{l \cap r} H_2$.\\
(2)$ \Longrightarrow $ (3)\\
Let $H_1\sim_{l \cap r} H_2$, so that there exist $g\in G^m$ and $f\in G^n$ such that
$$
H_1=H_2 \underline{g} \qquad H_2=\underline{f}^* H_1.
$$
Therefore, we have $H_1=\underline{f}^* H_1\underline{g}$, so that $(f,g)\in St_{l\times r} (H_1)$.\\
(2)$ \iff $(4) The proof is similar to the case (2)$ \iff $(3).
\end{proof}

\begin{remark}\label{rem:ab}
Combining Corollary \ref{cor:cap} with Proposition \ref{prop:st}, if $G$ is abelian,
for $H_1,H_2\in \mathcal H_\Gamma$, we deduce that $\Psi(H_1)=\Psi(H_2)$ and $\Psi_L(H_1)=\Psi_L(H_2)$ if and only if there exists $g\in G^m$ such that $H_2=H_1 \underline{g}$ with $g_1=g_2=\cdots=g_m$. By Eq. \eqref{ettore}, this is equivalent to the existence of $g_0\in G$ such that $H_1=H_2 g_0=g_0 H_2$.\\
\indent In other words, the classes of  $G$-phases up to multiplication by a ``scalar'' in $G$, are in bijection with the compatible pairs. In particular, for a given gain function $\psi$ of $\Gamma$, choosing a compatible gain function for the line graph $L(\Gamma)$ is equivalent to choose, up to a scalar multiplication, a $G$-phase in $\Psi^{-1}(\psi)$.
In the language of \cite{reff0}, when $G$ is abelian, an oriented $G$-gain graph $(\Gamma,\psi,H)$, considering $H$ up to scalar multiplications, is a $G$-gain graph $(\Gamma,\psi)$ where a compatible gain $\zeta$ is fixed for $L(\Gamma)$.
\end{remark}

\subsection{$G$-phases, switching equivalence and gain-line graph}
We are going to study how the switching equivalence affects our definitions and actions, with the aim to give a well posed definition of line graph.

As a first step, we investigate the relation between the action of an element  $f\in G^n$ (resp.\ $g\in G^m$) on $\mathcal H_{\Gamma}$ and the switching action of the associated
switching function
$$
f\colon V_{\Gamma}\to G, \;\;f(v_i)\mapsto f_i\qquad \mbox{(resp.\ } g\colon V_{L(\Gamma)}\to G, \;\; g(v_i)\mapsto g_i\mbox{)}
$$
on $G(\Gamma)$ (resp.\ on  $G(L(\Gamma))$), see Definition \ref{defswe}.

\begin{lemma}\label{lem:action}
Let $H\in \mathcal H_\Gamma$, $f\in G^n$ and $g\in G^m$. We  have:
\begin{equation*}
\begin{split}
\Psi(\underline{f}^*H\underline{g})&=\Psi(H)^f\\
\Psi_L(\underline{f}^*H\underline{g})&=\Psi_L(H)^g.
\end{split}
\end{equation*}
\end{lemma}
\begin{proof}
Let $e_k=\{v_i,v_j\} \in E_{\Gamma}$. Then, by Eq. \eqref{eq:H}, we have:
\begin{equation*}
\begin{split}
\Psi(\underline{f}^*H\underline{g})(v_i,v_j)&= s_1(\underline{f}^*H\underline{g})_{i,k} (\underline{f}^*H\underline{g})_{j,k}^{-1}=s_1(f_i^{-1}H_{i,k} g_k) (f^{-1}_j H_{j,k} g_k)^{-1}\\&=
s_1f_i^{-1}H_{i,k} g_k g_k^{-1}H_{j,k}^{-1} f_j=f_i^{-1} s_1H_{i,k} H_{j,k}^{-1}f_j=f_i^{-1} \Psi(H)(v_i,v_j) f_j\\&=\Psi(H)^f(v_i,v_j).
\end{split}
\end{equation*}
Let $E_l=\{e_i,e_j\}\in E_{L(\Gamma)}$, $v_k=e_i\cap e_j\in V_\Gamma$. Then, by Eq. \eqref{eq:LH}, we have:
\begin{equation*}
\begin{split}
\Psi_L(\underline{f}^*H\underline{g})(e_i,e_j)&= s_2(\underline{f}^*H\underline{g})_{k,i}^{-1} (\underline{f}^*H\underline{g})_{k,j}=
s_2 (f_k^{-1} H_{k,i} g_i)^{-1} (f_k^{-1}H_{k,j} g_j)\\&=s_2 g_i^{-1} H_{k,i}^{-1} f_kf_k^{-1}H_{k,j} g_j=g_i^{-1} s_2  H_{k,i}^{-1}H_{k,j} g_j=g_i^{-1} \Psi_L(H)(e_i,e_j) g_j\\&= \Psi_L(H)^g(e_i,e_j).
\end{split}
\end{equation*}
\end{proof}
The previous lemma, combined with the results of the previous section, allows us to enunciate one of the main results of the paper.

\begin{theorem}\label{th:rl}
The orbits of the action $l\times r$ on $\mathcal H_\Gamma$ are exactly the subsets $\Psi^{-1}([\psi])$ or, equivalently, the subsets $\Psi_L^{-1}([\zeta])$.
In other words, for $H_1,H_2\in\mathcal H_\Gamma$, the following are equivalent:
\begin{enumerate}
\item $H_1\sim_{l\times r} H_2$
\item $\Psi(H_1)\sim \Psi(H_2)$
\item $\Psi_L(H_1)\sim \Psi_L(H_2).$
\end{enumerate}
In particular, there exists an injection $\mathcal L\colon [G(\Gamma)]  \to [G(L(\Gamma))]$,
such that $\mathcal L([\Psi(H)])=[\Psi_L(H)]$, for every $H\in \mathcal H_\Gamma$.
\end{theorem}
\begin{proof}${}$\\
(1)$\iff $(2)\\
Suppose $H_1\sim_{l\times r} H_2$, so that $H_2=\underline{f}^* H_1\underline{g}$. By Lemma \ref{lem:action} it follows that
$$
\Psi(H_2)=\Psi(\underline{f}^* H_1\underline{g})=\Psi(H_1)^{f},
$$
that is $\Psi(H_1)\sim\Psi(H_2)$.
\\ Vice versa, if $\Psi(H_2)=\Psi(H_1)^f$, we define $H_0:=\underline{f}^* H_1\in \mathcal H_\Gamma$. By Lemma \ref{lem:action}, we have
$$
\Psi(H_0)=\Psi(\underline{f}^* H_1)=\Psi(H_1)^f=\Psi(H_2).
$$
By virtue of Theorem \ref{th:r}, we have $H_0\sim_r H_2$. But we also have, by definition of $H_0$,  $H_0\sim_l H_1$. It follows from Remark \ref{rem:transitive} that $H_1\sim_{l\times r} H_2$.\\
(1)$\iff $(3)\\
Suppose $H_1\sim_{l\times r} H_2$, so that $H_2=\underline{f}^* H_1\underline{g}$. By Lemma \ref{lem:action} it follows that
$$
\Psi_L(H_2)=\Psi_L(\underline{f}^* H_1\underline{g})=\Psi_L(H_1)^{g},
$$
that is $\Psi_L(H_1)\sim\Psi_L(H_2)$.
\\ Vice versa, if $\Psi_L(H_2)=\Psi_L(H_1)^g$, we define $H_0:= H_1 \underline{g} \in \mathcal H_\Gamma$. By Lemma \ref{lem:action}, we have
$$\Psi_L(H_0)=\Psi_L( H_1\underline{g})=\Psi_L(H_1)^g=\Psi_L(H_2).$$
 By virtue of Theorem \ref{th:l}, we have $H_0\sim_l H_2$. But we also have, by definition of $H_0$,  $H_0\sim_r H_1$. It follows from Remark \ref{rem:transitive} that $H_1\sim_{l\times r} H_2$.
\end{proof}

\begin{corollary}\label{cor:nuovo}
Let  $(\psi,\zeta)$ be a pair in $G(\Gamma)\times G(L(\Gamma))$. The following are equivalent:
\begin{enumerate}
\item $(\psi,\zeta)$ is compatible;
\item $\mathcal L([\psi])=[\zeta]$.
\end{enumerate}
In particular, if $(\psi,\zeta)$ is compatible then $(\psi',\zeta')$ is compatible for any $\psi'\sim \psi$ and any $\zeta'\sim\zeta$.
\end{corollary}
\begin{proof}
Let $(\psi,\zeta)=(\Psi(H),\Psi_L(H))$, for some $H\in \mathcal{H}_\Gamma$. By Theorem \ref{th:rl} we have $[\zeta]=[\Psi_L(H)]= \mathcal L([\Psi(H)])= \mathcal L([\psi])$.\\
\indent Vice versa,  suppose  $\mathcal L([\psi])=[\zeta]$.
By Proposition \ref{prop:sur}, there exists $H\in \mathcal H_\Gamma$ such that $\Psi(H)=\psi$.
Then $[\zeta]=\mathcal L([\psi])=\mathcal L([\Psi(H)])=[\Psi_L(H)]$. In particular, there exists $g\in G^m$ such that $\zeta=\Psi_L(H)^g$.
Then $(\psi,\zeta)= (\Psi(H),\Psi_L(H)^g) = (\Psi(H\underline{g}),\Psi_L(H\underline{g}))$ by Lemma~\ref{lem:action}.
\end{proof}
Corollary \ref{cor:nuovo} shows that the compatible pairs are exactly all the choices of representatives of $[\psi]$ and  $\mathcal L([\psi])$.
In particular, the properties introduced in Definition \ref{def:compatible}, like the property of being \emph{gain-line function} in $G(L(\Gamma))$ or the property of being \emph{compatible} in $G(\Gamma)\times G(L(\Gamma))$, do not depend on the choice of representatives of the switching classes.

Then we can interpret $\mathcal L([\psi])$ as the \emph{gain-line graph} of $[\psi]$, and  the map $\mathcal L$ is the line graph in the context of gain graphs up to switching equivalence.\\
\indent Now let $N_\Gamma(G)$ be the matrix defined in Example \ref{ex:N}.
\begin{corollary}\label{cor:bal}
For $H\in \mathcal H_\Gamma$, the following are equivalent:
\begin{enumerate}
\item $H\sim_{l\times r} N_\Gamma(G)$;
\item $\Psi(H)\sim \bold{s_1}$ on $\Gamma$, that is $[\Psi(H)]=[ \bold{s_1}]$;
\item $\Psi_L(H)\sim \bold{s_2}$ on $L(\Gamma)$, that is $[\Psi_L(H)]=[ \bold{s_2}]$.
\end{enumerate}
In particular, with the choice $s_1=s_2=1_G$, for a compatible pair
$(\psi,\zeta)\in G(\Gamma)\times G(L(\Gamma))$ one has that $(\Gamma,\psi)$ is balanced if and only if $(L(\Gamma),\zeta)$ is balanced.
\end{corollary}
\begin{proof}
Notice that, by Definition \ref{eq:psi}, we have $\Psi(N_\Gamma(G))=\bold{s_1}$ and, by Definition \ref{def:Lpsi}, we have $\Psi_L(N_\Gamma(G))=\bold{s_2}$. Then the equivalences follow from Theorem \ref{th:rl}.
\end{proof}

\begin{remark}
When $G=\mathbb{T}_4$, with the choices $s_1=-1$ and $s_2=1$, our Corollary \ref{cor:bal} implies that the  switching equivalence class of the line graph is balanced if and only if the switching equivalence class of the starting graph is \emph{antibalanced}. This is
the content of \cite[Theorem~2]{francesco}.\\
\indent Moreover, when $G=\{\pm1\}$ with $s_1=s_2=-1$, our Corollary \ref{cor:bal} implies that the switching equivalence class  of the line graph is balanced if and only if the switching equivalence class  of the starting graph is balanced, as one can deduce from \cite[Section~V]{zasmat}.
\end{remark}

We have seen that the map $\mathcal L\colon [G(\Gamma)]\to [G(L(\Gamma))]$ is defined up to switching equivalence. However, paying the price to make a particular choice by fixing an orientation $\mathfrak o$ of $\Gamma$, we can \emph{lift} the map $\mathcal L$ to a map $\mathcal L_{\mathfrak o}\colon  G(\Gamma)\to G(L(\Gamma))$.
This is possible thanks to the map $H_{\mathfrak o}$, introduced in Definition \ref{H0}, that in Proposition \ref{prop:sur} is proved to be a \emph{section} of $\Psi\colon \mathcal H_{\Gamma}\to G(\Gamma)$. By putting $\mathcal L_{\mathfrak o}:=\Psi_L\circ H_{\mathfrak o}$ we have that, by Theorem \ref{th:rl}, the following diagram commutes.
$$
\begin{tikzcd}
G(\Gamma) \arrow[r,"\mathcal L_{\mathfrak o}" ]  \arrow[d,two heads]  &       G(L(\Gamma)) \arrow[d,two heads]\\
\left[G(\Gamma)\right] \arrow[r,"\mathcal L" ]       &           \left[G(L(\Gamma))\right]\\
\end{tikzcd}$$
For a more complete picture of the relations among these maps see Fig. \ref{fig:uovo}.
We are also able to give in Proposition \ref{nutella} an explicit description of the map $\mathcal L_{\mathfrak o}$.
In particular, the rules given in Eq. \eqref{eq:yu} are represented in Fig. \ref{fig:lift}.

\begin{proposition}\label{nutella}
Let $\psi$ be a gain function on $\Gamma$, and let $\mathfrak o$ be a fixed orientation for its edges.
Let $e_a,e_b\in E_\Gamma$ such that $e_a=\{v_1,v_2\}$, $e_b=\{v_2,v_3\}$, so that $\{e_a,e_b\}\in E_{L(\Gamma)}$. Then:
\begin{equation}\label{eq:yu}
\mathcal L_{\mathfrak o}(\psi)(e_a,e_b)=\begin{cases}
s_1s_2 \psi(v_2,v_3) &\mbox{ if } e^{\mathfrak o}_a=(v_1,v_2), e^{\mathfrak o}_b=(v_2,v_3)\\
s_2 &\mbox{ if } e^{\mathfrak o}_a=(v_1,v_2),e^{\mathfrak o}_b=(v_3,v_2)\\
 s_2 \psi(v_1,v_2)  \psi(v_2,v_3)&\mbox{ if } e^{\mathfrak o}_a=(v_2,v_1),e^{\mathfrak o}_b=(v_2,v_3).
\end{cases}
\end{equation}
\end{proposition}
\begin{proof}
For $\{e_a,e_b\}\in E_{L(\Gamma)}$, $v_2= e_a\cap e_b$, by Definition \ref{def:Lpsi} we have:
\begin{equation}\label{eq:li}
\mathcal{L}_{\mathfrak o}(\psi)(e_a,e_b)=\Psi_L(H_{\mathfrak o}(\psi))(e_a,e_b)=s_2 H_{\mathfrak o}(\psi)_{2,a}^{-1} H_{\mathfrak o}(\psi)_{2,b}.
\end{equation}
Since $e_a=\{v_1,v_2\}$, $e_b=\{v_2,v_3\}$, by Definition \ref{H0} we have:
\begin{equation}\label{eq:oo}
H_{\mathfrak o}(\psi)_{2,a}:=
\begin{cases}
\psi(v_2,v_1) &\mbox{ if } (v_2,v_1)\in \mathfrak o\\
s_1 &\mbox{ if } (v_1,v_2)\in \mathfrak o
\end{cases}\quad H_{\mathfrak o}(\psi)_{2,b}:=
\begin{cases}
\psi(v_2,v_3) &\mbox{ if } (v_2,v_3)\in \mathfrak o\\
s_1 &\mbox{ if } (v_3,v_2)\in \mathfrak o.
\end{cases}
\end{equation}
By gluing together Eqs. \eqref{eq:li} and \eqref{eq:oo}, we obtain Eq. \eqref{eq:yu} and the proof is completed.
 \end{proof}

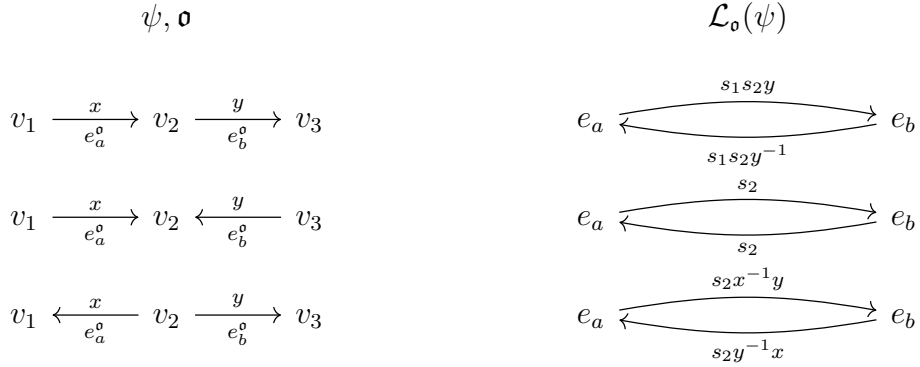
\begin{figure}
\begin{center}
$$
\begin{tikzcd}[scale=0.9]
&&\psi, \mathfrak o& &&&  &\mathcal L_{\mathfrak o}(\psi)& \\
&v_1\arrow[r,"x","e_a^{\mathfrak o}"']&v_2 \arrow[r,"y","e_b^{\mathfrak o}"']&v_3&&&{e_a} \arrow[rr,"s_1s_2 y",bend left=10] &&{e_b} \arrow[ll,"s_1s_2 y^{-1}",bend left=10]\\
&v_1\arrow[r,"x","e_a^{\mathfrak o}"']&v_2&v_3  \arrow[l,"y"',"e_b^{\mathfrak o}"]&&&{e_a} \arrow[rr,"s_2",bend left=10] &&{e_b} \arrow[ll,"s_2",bend left=10]\\
&v_1&v_2 \arrow[l,"x"',"e_a^{\mathfrak o}"]\arrow[r,"y","e_b^{\mathfrak o}"']&v_3&&&{e_a} \arrow[rr,"s_2x^{-1} y",bend left=10] && {e_b} \arrow[ll,"s_2 y^{-1} x",bend left=10]
\end{tikzcd}
$$
\end{center}\caption{The construction of  $\mathcal L_{\mathfrak o}(\psi)$.}\label{fig:lift}
\end{figure}
\begin{example}
Consider the Paw graph $P$, its $\mathbb{Q}_8$-gain function $\psi$ and its orientation $\mathfrak o_<$ described in Example \ref{ex:paw}, with the choices $s_1=s_2=-1$.
On the left side of Fig. \ref{fig:pawline} the gain graph $(P,\psi)$ is depicted with the orientation of the edges given by $\mathfrak{o}_{<}$.
On the right side of Fig. \ref{fig:pawline} the associated gain-line graph $(L(P),\mathcal L_{\mathfrak{o}_{<}}(\psi))$ is depicted. Notice how the gain function $\mathcal L_{\mathfrak{o}_{<}}(\psi)$ can be constructed via the rules described in Fig. \ref{fig:lift}.
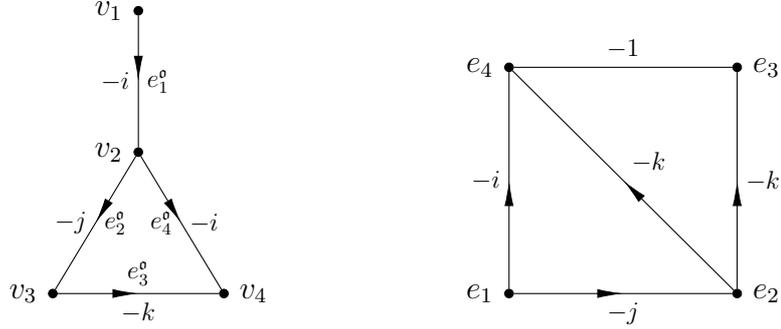
\begin{figure}
\begin{center}
\begin{tikzpicture}[scale=0.75]

\begin{scope}
[decoration={
    markings,
    mark=at position 0.5 with {\arrow{Latex[length=3.4mm, width=1.2mm]}}}
    ]
\tikzset{vertex/.style = {draw,circle,fill=black,minimum size=4pt,
                            inner sep=0pt}}
\tikzset{uedge/.style = {-,> = latex'}}

\node[vertex] (v3) at  (0,0) [scale=0.8,label=left:$v_3$]{};
\node[vertex] (v4) at  (3,0) [scale=0.8,label=right:$v_4$]{};
\node[vertex] (v2) at  (1.5,2.5) [scale=0.8,label=left:$v_2$]{};
\node[vertex] (v1) at  (1.5,5) [scale=0.8,label=left:$v_1$]{};

\node[vertex] (e1) at  (8,0) [scale=0.8,label=left:$e_1$]{};
\node[vertex] (e2) at  (12,0) [scale=0.8,label=right:$e_2$]{};
\node[vertex] (e3) at  (12,4) [scale=0.8,label=right:$e_3$]{};
\node[vertex] (e4) at  (8,4) [scale=0.8,label=left:$e_4$]{};

\draw[postaction={decorate}]  (v1) -- (v2)node[scale=0.8,midway,left] {$-i$}node[scale=0.8,midway,right] {$e^{\mathfrak o}_1$};
\draw[postaction={decorate}]  (v2) -- (v3)node[scale=0.8,midway,left] {$-j$}node[scale=0.8,midway,right] {$e^{\mathfrak o}_2$};
\draw[postaction={decorate}]  (v3) -- (v4)node[scale=0.8,midway, below] {$-k$}node[scale=0.8,midway,above] {$e^{\mathfrak o}_3$};
\draw[postaction={decorate}] (v2) -- (v4)node[scale=0.8,midway, right] {$-i$}node[scale=0.8,midway,left] {$e^{\mathfrak o}_4$};

\draw[postaction={decorate}]  (e1) -- (e2)node[scale=0.8,midway,below] {$-j$};
\draw[postaction={decorate}]  (e2) -- (e3)node[scale=0.8,midway,right] {$-k$};
\draw (e3) -- (e4)node[scale=0.8,midway,above] {$-1$};
\draw[postaction={decorate}]  (e1) -- (e4)node[scale=0.8,midway,left] {$-i$};
\draw[postaction={decorate}]  (e2) -- (e4)node[scale=0.8,midway,above right] {$-k$};
\end{scope}

\end{tikzpicture}
\end{center}\caption{The gain graph $(P,\psi)$ and its gain-line graph $(L(P),\mathcal{L}_{\mathfrak{o}_{<}}(\psi))$ associated with the orientation $\mathfrak o_<$.}  \label{fig:pawline}
\end{figure}
\end{example}

Observe that, when $G=\{\pm1\}$, with $s_1=s_2=-1$, this construction of signed line graphs is consistent, up to switching equivalence, with the one outlined in \cite{zasmat}.\\
\indent Moreover, when $G=\mathbb{T}_4$ with $s_1=-1$ and $s_2=1$, the line $\mathbb T_4$-gain graph that we obtain is switching equivalent to that in \cite[Eq.~(5)]{francesco}.

\subsection{Line $G$-phased graph}\label{line}
The aim of this section is to compare our definition of line graph with the definition of line graph of an oriented $G$-gain graph given in \cite{reff0}.
More precisely, Reff defined a map from the set of $G$-phases of $\Gamma$ to the set of $G$-phases of $L(\Gamma)$, and he proved that it is well defined with respect to the switching equivalence of the gain functions obtained from the $G$-phases (see \cite[Theorem~4.2]{reff0}). We are going to show that our machinery is consistent with that of \cite{reff0} and generalizes the result of Reff to the nonabelian case.

Recall that $L(\Gamma)=(V_{L(\Gamma)},E_{L(\Gamma)})$, with $V_{L(\Gamma)}=E_\Gamma$ and fix for $V_{L(\Gamma)}$ the same order of $E_\Gamma$.
Suppose $|E_{L(\Gamma)}|=q$ and let us fix an order $E_{L(\Gamma)}=\{E_1,\ldots,  E_q\}$. For an edge $E_k=\{e_i,e_j\}\in E_{L(\Gamma)}$, let us set $v(E_k):=e_i\cap e_j\in V_\Gamma$.\\
\indent Notice that $\mathcal{H}_\Gamma\subseteq M_{n\times m}(\mathbb C G)$ and $\mathcal{H}_{L(\Gamma)}\subseteq M_{m\times q}(\mathbb C G)$.
\begin{definition}\label{def:L}
Put $L\colon \mathcal H_\Gamma \to \mathcal H_{L(\Gamma)}$ such that
\begin{equation*}
L(H)_{i,k}=\begin{cases}
0 &\mbox{ if } e_i\notin E_k\\
H_{l,i}^{-1} &\mbox{ if } e_i\in E_k, v(E_k)=v_l.\\
\end{cases}
\end{equation*}
\end{definition}
Observe that our Definition \ref{def:L} coincides with the definition given in \cite[Eq.~(4.1)]{reff0}.\\
\noindent With a given $G$-phased graph $(\Gamma,H)$ we can associate the \emph{$G$-phased-line graph $(L(\Gamma),L(H))$}.
Among the results of the next theorem there is that the pair $(\Psi(H),s_1s_2 \Psi(L(H)))\in G(\Gamma)\times G(L(\Gamma))$ is compatible, or equivalently $\mathcal L([\Psi(H)])=[s_1s_2 \Psi(L(H))]$.
\begin{theorem}\label{prop:line}
For any $H\in \mathcal{H}_\Gamma$, for any $g\in G^m$, for any $f\in G^n$, one has:
\begin{enumerate}
\item
$L(H\underline{g})=\underline{g}^* L(H)$;
\item
$L(\underline{f}^*H)=L(H)\underline{f'}$, for some $f'\in G^q$;
\item
$\Psi_L(H)=s_1s_2 \Psi(L(H))$.
\end{enumerate}
\end{theorem}
\begin{proof}${}$\\
(1)\\
If $e_i\notin E_k$, clearly $L(H\underline{g})_{i,k}=0=(\underline{g}^* L(H))_{i,k}$.
Now suppose that $e_i\in E_k$ and $v(E_k)=v_l$. Then:
\begin{equation*}
L(H\underline{g})_{i,k}=(H\underline{g})_{l,i}^{-1}= (H_{l,i}g_i)^{-1}= g_i^{-1}H_{l,i}^{-1}=g_i^{-1} L(H)_{i,k}=(\underline{g}^* L(H))_{i,k}
\end{equation*}
and the thesis follows.\\
(2)\\
For every $k=1,\ldots,q$, there exists $l_k\in\{1,\ldots,n\}$ such that $v(E_k)=v_{l_k}$.
For a given $f\in G^n$, let us define the element $f'\in G^q$ such that $f'_k=f_{l_k}$.
If  $e_i\notin E_k$, clearly $L(\underline{f}H)_{i,k}=0=(L(H)\underline{f'})_{i,k}$.
Now suppose that $e_i\in E_k$ and $v(E_k)=v_{l_k}$. Then
\begin{equation*}
L(\underline{f}^*H)_{i,k}=(\underline{f}^*H)_{l_k,i}^{-1}=(f_{l_k}^{-1} H_{l_k,i})^{-1}= H_{l_k,i}^{-1} f_{l_k}= L(H)_{i,k} f'_k=(L(H)\underline{f'})_{i,k}
\end{equation*}
and the thesis follows.\\
(3)\\
From Definition \ref{eq:psi} we have
$\Psi(L(H))(e_i,e_j)=s_1 L(H)_{i,k} (L(H)_{j,k})^{-1}$
with $E_k=\{e_i,e_j\}\in E_{L(\Gamma)}$. Suppose that $v_l=e_i\cap e_j \in V_\Gamma$. Then by Definition \ref{def:L} we have
$L(H)_{i,k}=H_{l,i}^{-1}$ and $L(H)_{j,k}=H_{l,j}^{-1}$, so that:
$$
\Psi(L(H))(e_i,e_j)=s_1 L(H)_{i,k} (L(H)_{j,k})^{-1}=s_1 H_{l,i}^{-1} H_{l,j}.
$$
On the other hand, by Definition \ref{def:Lpsi}, we also have $\Psi_L(H)(e_i,e_j)=s_2 H_{l,i}^{-1} H_{l,j}$ and the thesis follows.
\end{proof}
As a consequence, when $s_1=s_2$ and the interest is only in the gain-line graph and not in its $G$-phase, considering $\Psi(L(H))$ or $\Psi_L(H)$ is equivalent. Under these assumptions, our results in the previous sections can be applied in the context of \cite{reff0}.
More precisely,  when $G$ is abelian with $s_1=s_2$, Corollary \ref{cor:bal}, combined with Theorem \ref{prop:line}, generalizes the content of \cite[Proposition~4.1]{reff0}. Moreover Theorem \ref{th:rl}, combined with Theorem \ref{prop:line}, gives a generalization of the content of \cite[Theorem~4.2]{reff0}. In particular, we give a complete answer to \cite[Question~2]{reff0}.

\begin{figure}
\begin{center}
\begin{tikzcd}[row sep=small]
&&&G(\Gamma)\arrow[ddr,two heads,bend left=5]  \\
&&&\\
& \mathcal H_\Gamma / \sim_{l\cap r} \arrow[lddd,"{\tiny _{Cor.\ref{cor:cap}} }",hook]&&& \left[G(\Gamma)\right] \ar[dddddd, hook', "\mathcal L"  description,"{\tiny \; \; _{Thm. \ref{th:rl}} }", dashed, bend left=45]\\
&&&\\
&&& \mathcal H_\Gamma/\sim_{r} \arrow[uuuu,hook,two heads] \arrow[dr,two heads]\\
G(\Gamma)\times G(L(\Gamma)) \arrow[rrruuuuu,two heads, bend left=40]  \arrow[rrrddddd,two heads, bend right=40]&&\mathcal H_\Gamma  \arrow[luuu,two heads]   \arrow[ll,"\Psi\times\Psi_{L}"] \arrow[ur,two heads]
\arrow[from=uuuuur, "{\tiny _{H_{\mathfrak o}}}" description, bend right=50,dotted]
 \arrow[uuuuur, "\Psi"  description, "{\tiny   _{Thm. \ref{th:r}} }"',"{\tiny _{Prop.\ref{prop:sur}}}",two heads] \arrow[dr,two heads] \arrow[dddddr,,"{\tiny ^{^{Thm. \ref{th:l}}}}","\Psi_L" description]  \arrow[lddd, "L" description]&  &\mathcal H_\Gamma/\sim_{l\times r} \arrow[uuu,hook,"{\tiny _{Thm. \ref{th:rl}}}",two heads] \arrow[ddd,"{\tiny _{Thm. \ref{th:rl}}}"',hook']& \\
&&&  \mathcal H_\Gamma/\sim_{l} \arrow[dddd,hook'] \arrow[ur,two heads] \\
&&&\\
&\mathcal H_{L(\Gamma)} \arrow[rrdd," s_1s_2\Psi"',"{\tiny ^{^{Thm. \ref{prop:line}}}}",two heads]&&&\left[G(L(\Gamma))\right]\\
&& &\\
&&& G(L(\Gamma)) \arrow[uur,two heads,bend right=5] \arrow[from=uuuuuuuuuu,dashed, "\mathcal L_{\mathfrak o}" description, bend left=149]
\end{tikzcd}
\end{center}\caption{Summary of Section \ref{sec:phases}.}\label{fig:uovo}
\end{figure}
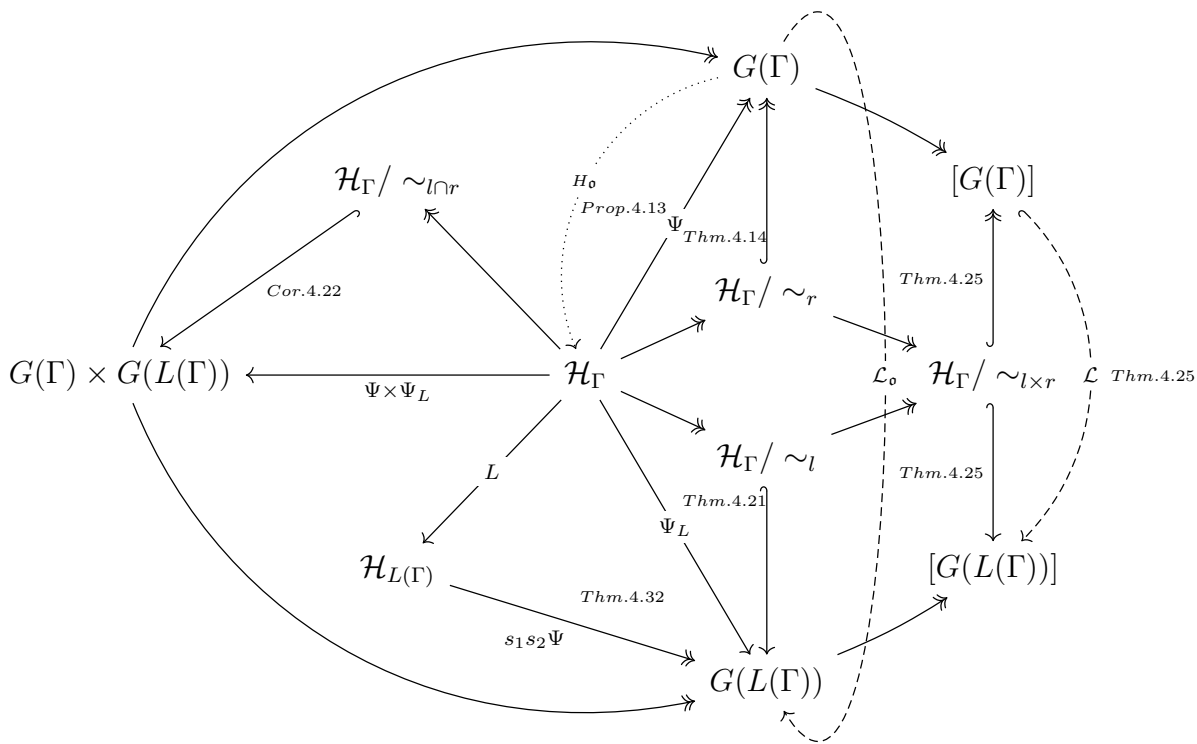

\section{Represented Adjacency, Laplacian and $G$-phase matrices}\label{sec:rep}

We start this section by recalling some basic definitions about group representations. The interested reader is referred to \cite{fulton} for further details.\\
 A \textit{representation} of $G$ on a vector space $V$ of dimension $k$ is a group homomorphism
$
\pi\colon G\to GL_k(V),
$
where $GL_k(V)$ is the group of all bijective $\mathbb{\mathbb C}$-linear maps from $V$ to itself. Notice that the group $GL_k(V)$ can be naturally identified with the group $GL_k(\mathbb{C})$ of all invertible matrices of size $k$ over $\mathbb{C}$. Therefore, with a small abuse of notation, one can denote by $\pi$ the homomorphism that associates with an element $g\in G$ the matrix in $GL_k(\mathbb{C})$ corresponding to $\pi(g)$. The dimension $k$ of $V$ is called the \emph{degree} of $\pi$. \\
\indent A representation $\pi$ is said to be \emph{unitary} if $\pi(g)\in U_k(\mathbb{C})$, for each $g\in G$, with $U_k(\mathbb{C}) = \{M\in GL_k(\mathbb{C}) : M^{-1}=M^\ast\}$, where $M^*$ denotes the Hermitian transpose of $M$; $\pi$ is said to be \textit{faithful} if only the neutral element of $G$ is sent to the identity matrix.\\
\indent A representation $\pi$ of $G$ on $V$ is said to be \textit{irreducible} if there is no proper invariant subspace $W\subset V$ under the action of $G$, that is, there is no $W$ such that $\pi(g)w\in W$ for any $w\in W$ and for any $g\in G$. It is known that each representation $\pi$ of a group $G$ can be decomposed into the direct sum of irreducible representations.

For any representation $\pi$ of $G$ of degree $k$, for any element $f\in \mathbb C G$, with $f=\sum_{x\in G} f_x x$, the \emph{Fourier transform of $f$ at the representation $\pi$} is the matrix $\hat{f}(\pi)\in M_k(\mathbb C)$ such that:
$$\hat{f}(\pi)=\sum_{x\in G} f_x \pi(x).$$
Moreover, with a given $A =(A_{ij})\in M_{n\times m}(\mathbb CG)$, we can associate a matrix $\widehat{A}(\pi)\in M_{nk\times mk }(\mathbb C)$, that we can write as a block matrix $n\times m$ of type
$$
\left(\begin{array}{c|c|c|c}
                      \widehat{A}(\pi)_{1,1} & \widehat{A}(\pi)_{1,2} & \cdots & \widehat{A}(\pi)_{1,m} \\\hline
                       \widehat{A}(\pi)_{2,1} & \widehat{A}(\pi)_{2,2} & \cdots & \widehat{A}(\pi)_{2,m} \\   \hline
                         \vdots & \vdots & \ddots  & \vdots \\ \hline
                       \widehat{A}(\pi)_{n,1} & \widehat{A}(\pi)_{n,2} & \cdots & \widehat{A}(\pi)_{n,m} \
                       \end{array}\right)
$$
where  the block  $\widehat{A}(\pi)_{i,j}\in M_{k\times k}(\mathbb C)$ is the Fourier transform of $A_{i,j}\in \mathbb C G$ at $\pi$.

In particular, starting from the adjacency matrix of a gain graph $A_{\Gamma,\psi}\in M_{n\times n}(\mathbb CG)$, the matrix  $\widehat{A_{\Gamma,\psi}}(\pi)\in M_{nk\times nk}(\mathbb C)$ is obtained by replacing each entry $g\in G$ with the $k\times k$ block given by $\pi(g)$, and each zero entry with a zero matrix of size $k$.
The matrix $\widehat{A_{\Gamma,\psi}}(\pi)$ is the \emph{represented adjacency matrix} $\widehat{A_{\Gamma,\psi}}(\pi)$ of the gain graph $(\Gamma,\psi)$ with respect to the representation $\pi$.\\
\indent The \emph{represented Laplacian matrix} $\widehat{\Delta^s_{\Gamma,\psi}}(\pi)$ of the gain graph $(\Gamma,\psi)$ with respect to the representation $\pi$ is defined in a similar way, starting from $\Delta^s_{\Gamma,\psi}$.

When the representation is faithful, the matrix  $\widehat{A_{\Gamma,\psi}}(\pi)$ contains all the information about the gain graph $(\Gamma,\psi)$. For example, if the spectrum of  $\widehat{A_{\Gamma,\psi}}(\pi)$ is known, then it is possible to establish whether $(\Gamma,\psi)$ is balanced or not (see \cite{nostro}).

Given $f_1,f_2\in \mathbb C G$ and a representation $\pi$ of $G$ of degree $k$, the following properties hold (see, for example, \cite{fulton}).
\begin{eqnarray} \label{eq:somma}
\widehat{(f_1+f_2)}(\pi)&=& \widehat{f_1}(\pi)+ \widehat{f_2}(\pi);
\end{eqnarray}
\begin{eqnarray}\label{eq:prodotto}
\widehat{f_1f_2}(\pi)&=& \widehat{f_1}(\pi) \widehat{f_2}(\pi);
\end{eqnarray}
moreover, if $\pi$ is unitary, then
\begin{equation}\label{eq:star}
\widehat{f^*}(\pi)=  \widehat{f}(\pi)^* \qquad \textrm{for each } f\in \mathbb C G.
\end{equation}

\begin{lemma}\label{lem:fou}
For every $A,B\in M_{n\times l}(\mathbb CG)$, $C\in M_{l\times m}(\mathbb CG)$, $g\in G$ and for every representation $\pi$ of $G$, one has:
\begin{enumerate}
\item $\widehat{(A+B)}(\pi)=\widehat{A}(\pi)+\widehat{B}(\pi)$
\item $\widehat{AC}(\pi)=\widehat{A}(\pi)\widehat{C}(\pi)$
\item $\widehat{gA}(\pi)=(I_n\otimes \pi(g)) \widehat{A}(\pi)$
\item $\widehat{Ag}(\pi)=\widehat{A}(\pi)(I_l\otimes \pi(g)).$
\end{enumerate}
Moreover, if $\pi$ is unitary, then $\widehat{A^*}(\pi)=\widehat{A}(\pi)^*$.
\end{lemma}
\begin{proof}
By Eq. \eqref{eq:somma} we get the following equation of $k\times k$ blocks:
$$
\widehat{(A+B)}(\pi)_{i,j}= \widehat{(A+B)_{i,j}}(\pi)=\widehat{(A_{i,j}+B_{i,j})}(\pi)=\widehat{A_{i,j}}(\pi)+\widehat{B_{i,j}}(\pi)=(\widehat{A}(\pi)+\widehat{B}(\pi))_{i,j},
$$
and then (1) is proved.
By virtue of Eqs. \eqref{eq:somma} and \eqref{eq:prodotto} we have the following equation of $k\times k$ blocks:
\begin{equation*}
\begin{split}
\widehat{AC}(\pi)_{i,j}&= \widehat{(AC)_{i,j}}(\pi)=\widehat{\sum_{k=1}^l A_{i,k}C_{k,j}}(\pi)=\sum_{k=1}^l  \widehat{A_{i,k}C_{k,j}}(\pi)=\sum_{k=1}^l  \widehat{A_{i,k}}(\pi) \widehat{C_{k,j}}(\pi)\\&=
(\widehat{A}(\pi)\widehat{C}(\pi))_{i,j},
\end{split}
\end{equation*}
and then (2) is proved.\\
\indent By noticing that $gA=diag(g,\ldots,g) A$ and  $A g = A\, diag(g,\ldots,g)$,  combining with  $\widehat{diag(g,\ldots,g)}(\pi)=(I_n\otimes \pi(g))$, also (3) and (4) are proved.\\
\indent Finally, if $\pi$ is unitary, from Eq. \eqref{eq:star} we get the following equation of $k\times k$ blocks:
$$
\widehat{A^*}(\pi)_{i,j}=\widehat{(A^*)_{i,j}}(\pi)=\widehat{(A_{j,i})^*}(\pi)=\widehat{(A_{j,i})}(\pi)^*=(\widehat{A}(\pi)^*)_{i,j}.
$$
\end{proof}
In particular, the represented adjacency matrix and the represented Laplacian matrix of a gain graph, with respect to a finite dimensional unitary representation $\pi$, are Hermitian matrices, with real spectra.
\begin{definition}
For a matrix $A$ we denote by $\sigma(A)$ the \emph{spectrum} of $A$, that is, the (multi)set of the eigenvalues of $A$.
For a $G$-gain graph $(\Gamma,\psi)$ and a unitary representation $\pi$ of $G$, the $\pi$-spectrum of $(\Gamma,\psi)$ is the spectrum $\sigma(\widehat{A_{\Gamma,\psi}}(\pi))$.
\end{definition}
By applying the Fourier transform to both sides of Eqs. \eqref{eq:r} and \eqref{eq:l}, we obtain a \emph{represented} version of Lemma \ref{lem:r} and Lemma \ref{lem:l}, concerning  represented adjacency and Laplacian matrices and the Fourier transform of a $G$-phase, that we call \emph{represented $G$-phase}.
Notice that, if $\pi_0$ is the $1$-dimensional trivial representation (that is, $\pi_0(g)=1$, for every $g\in G$), a represented $G$-phase is nothing but the classical incidence matrix $N$; in this case, Eqs. \eqref{eq:r} and \eqref{eq:l} reduce to Eqs. \eqref{eq:1} and \eqref{eq:2}, respectively.

\begin{remark}\label{rem:ponte}
If $G$ is (a subgroup of) $\mathbb T$ we always have a canonical isomorphism
$$
\pi_{id}: \mathbb T \longrightarrow U_1(\mathbb C),
$$
associating the $1\times 1$ matrix $(z)$ with each $z\in \mathbb{T}$: in particular, the map $\pi_{id}$ is a faithful unitary representation of degree $1$.\\
\indent For a $\mathbb T$-gain graph $(\Gamma,\Psi)$, with adjacency matrix $A_{\Gamma,\psi}\in M_n(\mathbb C \mathbb T)$ and Laplacian matrix $\Delta^s_{\Gamma,\psi}\in  M_n(\mathbb C \mathbb T)$, the matrices $\widehat{A_{\Gamma,\psi}}(\pi_{id})$ and $\widehat{\Delta^s_{\Gamma,\psi}}(\pi_{id})$ correspond to the classical adjacency and (signless) Laplacian matrix of a complex unit gain graph (in the sense of \cite{reff1}).\\
\indent When $G=\{\pm 1\}$, a represented $G$-phase $\widehat{H}(\pi_{id})$ (with $s_1=-1$) is the incidence matrix (or bidirection) in the sense of \cite[Section~IV]{zasmat}.\\
\indent When $G=\mathbb T_4$ (with $s_1=-1$ and $s_2=1$), the matrix $\widehat{H}(\pi_{id})$ is the incidence matrix in the sense of \cite{francesco}.\\
\indent Finally when $G=\mathbb T$ (with $s_1=s_2$), the matrix $\widehat{H}(\pi_{id})$ is the incidence matrix in the sense of \cite[Section~5.1]{reff0}.\\
\indent For all these reasons, our theory is consistent with (and generalizes) the one in \cite{francesco, reff0, zasmat}.
\end{remark}

\subsection{Spectral conditions to be a gain-line graph}
As we have already observed, the Fourier transform provides a represented version of Lemma \ref{lem:l}. In this section, we develop our group representation approach in order to obtain some necessary conditions for a gain graph to be a gain-line graph (see Definition~\ref{def:compatible}).

\begin{theorem}\label{th:-1}
Let $H\in \mathcal H_\Gamma$ and let $\pi$ be a unitary representation of $G$ of degree $k$. Then:
$$
\widehat{A_{L(\Gamma),\Psi_L(H)}}(\pi)=\ \left(I_m\otimes \pi(s_2)\right) \left(\widehat{H}(\pi)^*\widehat{H}(\pi)-2 I_{km}\right).
$$
\end{theorem}
\begin{proof}
By Lemma \ref{lem:l} we have:
$s_2 A_{L(\Gamma), \Psi_L(H)}=H^*H-2\cdot \underline{ 1_{G^m} }$. Then Lemma \ref{lem:fou} implies:
\begin{eqnarray*}
\left(I_m\otimes\pi(s_2)\right)\widehat{A_{L(\Gamma),\Psi_L(H)}}(\pi) &=&\widehat{(s_2 A_{L(\Gamma),\Psi_L(H)})}(\pi)=\widehat{(H^*H-2\cdot \underline{ 1_{G^m} })}(\pi)\\
&=&\widehat{H}(\pi)^*\widehat{H}(\pi)-2 I_{km}.
\end{eqnarray*}
We conclude by using the fact that $(I_m\otimes\pi(s_2))^{-1}=I_m\otimes\pi(s_2)$ since $s_2$ is an involution.
\end{proof}

Let $\pi$ be an irreducible representation of degree $k$. Since $s_2$ is central, as a consequence of the Schur's Lemma (see \cite{fulton}), the matrix  $\pi(s_2)$ is a scalar matrix. Moreover, it must be $\pi(s_2)^2=I_k$, since $s_2$ is an involution. This implies that either $\pi(s_2)=I_k$ or $\pi(s_2)=-I_k$. The two cases are treated in Corollary \ref{cor:1} and Corollary \ref{cor:2}, respectively.\\

\begin{corollary}\label{cor:1}
Let $\pi$ be a unitary representation of $G$ of degree $k$ such that $\pi(s_2)=I_k$. Then, for every $H\in \mathcal H_\Gamma$:
$$
\sigma(\widehat{A_{L(\Gamma), \Psi_L(H)}}(\pi))\subseteq [-2,\infty).
$$
In particular, if $s_2=1_G$, the $\pi$-spectrum of a gain-line  graph is contained in  $[-2,\infty)$ for any unitary representation $\pi$.
\end{corollary}
\begin{proof}
From Theorem \ref{th:-1} we have $\widehat{A_{L(\Gamma), \Psi_L(H)}}(\pi)=\widehat{H}(\pi)^*\widehat{H}(\pi)-2 I_{km}$
and then
$$
\sigma(\widehat{A_{L(\Gamma), \Psi_L(H)}}(\pi))=\sigma( \widehat{H}(\pi)^*\widehat{H}(\pi))-2.
$$
The thesis follows since the matrix $\widehat{H}(\pi)^*\widehat{H}(\pi)$ is positive semidefinite.
\end{proof}
\begin{corollary}\label{cor:2}
Let $\pi$ be a unitary representation of $G$ of degree $k$ such that $\pi(s_2)=-I_k$. Then, for every $H\in \mathcal H_\Gamma$:
$$
\sigma(\widehat{A_{L(\Gamma), \Psi_L(H)}}(\pi))\subseteq (-\infty,2].
$$
\end{corollary}
\begin{proof}
From Theorem \ref{th:-1} we have $\widehat{A_{L(\Gamma), \Psi_L(H)}}(\pi)=-\widehat{H}(\pi)^*\widehat{H}(\pi)+2 I_{km}$
and then
$$
\sigma(\widehat{A_{L(\Gamma), \Psi_L(H)}}(\pi))=-\sigma( \widehat{H}(\pi)^*\widehat{H}(\pi))+2.
$$
The thesis follows since the matrix $\widehat{H}(\pi)^*\widehat{H}(\pi)$ is positive semidefinite.
\end{proof}
As we observed in Remark~\ref{rem:ponte}, when $G=\mathbb T$ the classical spectrum of a $\mathbb T$-gain graph is exactly its $\pi_{id}$-spectrum.
When $s_2=1$ clearly we have $\pi_{id}(s_2)=1$ and our Corollary \ref{cor:1} generalizes \cite[Theorem~5.5]{reff0}.
When $s_2=-1$ clearly we have $\pi_{id}(s_2)=-1$ and our Corollary \ref{cor:2} generalizes \cite[Theorem~5.4]{reff0}.

\begin{corollary}\label{cor:gain-line}
Let $\pi$ be a unitary irreducible representation of $G$. A $G$-gain graph whose $\pi$-spectrum contains eigenvalues less than $-2$ and eigenvalues greater than $2$ cannot be a gain-line graph, whatever the choice of $s_2$ is.
\end{corollary}
In the next example, we show that the conditions obtained in the previous corollaries are not trivial, even if we assume that the underlying graph is a line graph.

\begin{example}\label{ex:qua}
Let $G=\mathbb{Q}_8$ and consider its (unitary and irreducible) representation $\pi$ in Table \ref{tableQ8}.
\begin{table}
\begin{tabular}{|c|c|c|c|}
\hline
$\pm 1$ & $\pm i$ & $\pm j$ & $\pm k$ \\
\hline
$\pm \left(
   \begin{array}{cc}
     1 & 0 \\
     0 & 1  \\
   \end{array}
 \right)$  &  $\pm\left(
   \begin{array}{cc}
     0 & -1 \\
     1 & 0  \\
   \end{array}
 \right)$ & $\pm \left(
   \begin{array}{cc}
     0 & i \\
     i & 0  \\
   \end{array}
 \right)$ & $\pm \left(
   \begin{array}{cc}
     -i & 0 \\
     0 & i  \\
   \end{array}
 \right)$  \\
\hline
\end{tabular}
\smallskip
\caption{The representation $\pi$ of $\mathbb{Q}_8$ of Example \ref{ex:qua}.}\label{tableQ8}
\end{table}
Let $(\Gamma,\psi)$ be the gain graph represented in Fig. \ref{fig:qline}.
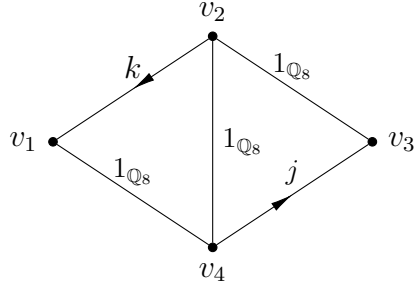
\begin{figure}
\begin{center}
\begin{tikzpicture}[scale=0.70]

\begin{scope}
[decoration={
    markings,
    mark=at position 0.5 with {\arrow{Latex[length=3mm, width=1.2mm]}}}
    ]
\tikzset{vertex/.style = {draw,circle,fill=black,minimum size=4pt,
                            inner sep=0pt}}
\tikzset{uedge/.style = {-,> = latex'}}

\node[vertex] (a) at  (0,0) [scale=0.8,label=left:$v_1$]{};
\node[vertex] (b) at  (3,2) [scale=0.8,label=above:$v_2$]{};
\node[vertex] (c) at  (6,0) [scale=0.8,label=right:$v_3$]{};
\node[vertex] (d) at  (3,-2) [scale=0.8,label=below:$v_4$]{};

\draw[postaction={decorate}] (b) --  (a) node[midway, above] {$k$};
\draw[uedge]  (b) -- (d)node[scale=0.9,midway,right] {$1_{\mathbb Q_8}$};
\draw[uedge] (b) -- (c)node[scale=0.9,midway,above] {$1_{\mathbb Q_8}$};
\draw[uedge] (a) -- (d)node[scale=0.9,midway, above] {$1_{\mathbb Q_8}$};
\draw[postaction={decorate}] (d) -- (c)node[scale=0.9,midway, above] {$j$};
\end{scope}
\end{tikzpicture}
\end{center}\caption{The $\mathbb{Q}_8$-gain graph $(\Gamma,\psi)$.}\label{fig:qline}
\end{figure}
We have:
$$
A_{\Gamma,\psi}=\left(
   \begin{array}{cccc}
     0 & -k & 0 & 1 \\
     k & 0 & 1 & 1  \\
       0 & 1 & 0 & -j  \\
         1 & 1 & j & 0  \\
   \end{array}\right);\qquad \widehat{A_{\Gamma,\psi}}(\pi)=\left(
   \begin{array}{cccccccc}
     0 & 0 & i & 0 & 0 & 0 & 1 & 0 \\
     0 & 0 & 0 & -i & 0 & 0 & 0 & 1  \\
      -i & 0 & 0 & 0 & 1 & 0 & 1 & 0  \\
      0 & i & 0 & 0 & 0 & 1 & 0 & 1\\
      0 & 0 & 1 & 0 & 0 & 0 & 0 & -i\\
      0 & 0 & 0 & 1 & 0 & 0 & -i & 0\\
      1 & 0 & 1 & 0 & 0 & i & 0 & 0\\
      0 & 1 & 0 & 1 & i & 0 & 0 & 0\\
         \end{array}\right).
$$
An explicit computation gives $\sigma(\widehat{A_{\Gamma,\psi}}(\pi))=\left\{ \left(\pm \frac{1}{2} \sqrt{10\pm 2\sqrt{17}}\right)^{(2)} \right\}$, where the exponent of each eigenvalue denotes its multiplicity. Since $\frac{1}{2} \sqrt{10+ 2\sqrt{17}} \sim 2.135779$, we have:
$$
\sigma(\widehat{A_{\Gamma,\psi}}(\pi))\cap (-\infty,-2)\neq \emptyset\quad \mbox{ and } \quad \sigma(\widehat{A_{\Gamma,\psi}}(\pi))\cap (+2,\infty)\neq \emptyset.
$$
As a consequence of Corollary \ref{cor:gain-line}, the gain graph $(\Gamma,\psi)$ cannot be a gain-line graph, neither for $s_2=1$ nor for $s_2=-1$, although the underlying graph $\Gamma$ is the line graph of the Paw graph.
\end{example}

\end{document}